\documentclass[11pt]{amsart}
\usepackage{graphicx}
\usepackage{tikz}
\usepackage{float}
\usepackage{mathrsfs}
\usepackage{color}
\usepackage{amssymb}
\usepackage{amsmath}
\usepackage[margin=1.4in]{geometry}
\usepackage[hyperindex,breaklinks]{hyperref}
\hypersetup{colorlinks=true,linkcolor=blue,citecolor=blue,urlcolor=blue,
 pdftitle={Index-one capillary minimal hypersurfaces in R\string^n\string_+}}

\newtheorem{prop}{Proposition}
\newtheorem{theo}{Theorem}
\renewcommand{\thetheo}{\arabic{theo}}
\newtheorem*{theo*}{Theorem}
\newtheorem{lemm}[prop]{Lemma}

\theoremstyle{plain}

\theoremstyle{definition}

\DeclareMathOperator{\Ind}{Ind}
\DeclareMathOperator{\Vol}{Vol}

\title[Index-one capillary minimal hypersurfaces in \(\mathbb R^n_+\)]{%
Index-one capillary minimal hypersurfaces in \(\mathbb R^n_+\)}
\author{Chenghang Lu}
\address{(C.L.)Albert-Ludwigs-Universit\"at Freiburg,
Mathematisches Institut, Ernst-Zermelo-Str.~1,
D-79104 Freiburg, Germany}
\email{chenghang.lu@math.uni-freiburg.de}
\thanks{C.L. is supported by the DFG through
IRTG 3132 (Grant No.~545046569).}
\date{\today}

\begin{document}
\begin{abstract}
We classify complete, two-sided, properly embedded capillary minimal
surfaces with nonempty boundary and Morse index one in
\(\mathbb R^3_+\). More precisely,
if the boundary is compact, these surfaces are catenoid truncations;
if the boundary is noncompact, they are capillary hairpin surfaces.
The classification extends to higher dimensions under extrinsic
Euclidean volume growth and Wang--Zhang's dimension and angle conditions.
\end{abstract}
\maketitle
\section{Introduction}
Consider a complete, two-sided, properly embedded capillary minimal
hypersurface \(M^n\subset\mathbb R^{n+1}_+\), \(n\ge2\), with nonempty boundary
and contact angle
\(\theta\in(0,\pi)\), where
\(\mathbb R^{n+1}_+=\{x\in\mathbb R^{n+1}:x_{n+1}\ge0\}\).
For \(f\in C_c^\infty(M)\), we write
\begin{equation*}
 Q(f,f)=\int_M\bigl(|\nabla f|^2-|h|^2f^2\bigr)\,dA
 -\int_{\partial M}qf^2\,ds,\qquad q=\cot\theta\,h(\mu,\mu).
\end{equation*}
Here \(h\) is the second fundamental
form and \(\mu\) is the outward unit conormal along \(\partial M\).
The space \(C_c^\infty(M)\) consists of smooth compactly supported
functions on \(M\), whose support may meet \(\partial M\).
We write \(Q(f)=Q(f,f)\) and use \(Q(f,g)\) for the associated symmetric
bilinear form. The  capillary Morse index is defined by
\begin{equation*}
 \Ind_\theta(M)=\sup\{\dim V:
 Q(f,f)<0\ \text{for every }0\ne f\in V\},
\end{equation*}
where \(V\) ranges over finite-dimensional subspaces of \(C_c^\infty(M)\).

We say that \(M\) has finite total curvature if
\(\int_M|h|^n\,dA<\infty\), and extrinsic Euclidean volume growth if
\(\Vol(M\cap B_R(0))\le CR^n\) for \(R\ge1\), where \(B_R(0)\) is an
ambient Euclidean ball. We prove the following results.

\begin{theo}\label{index-one-surface}
Let \(M^2\subset\mathbb R^3_+\) be a connected, complete, two-sided,
properly embedded capillary minimal surface with nonempty boundary and contact angle
\(\theta\in(0,\pi)\) and \(\Ind_\theta(M)=1\).
Then, up to dilations and rigid motions,
\begin{enumerate}
\item if \(\partial M\) is compact, then \(M\) is a truncation of the catenoid;
\item if \(\partial M\) is noncompact, then \(M\) is the capillary hairpin surface.
\end{enumerate}
\end{theo}
We extend this classification to higher dimensions under
extrinsic Euclidean volume growth and the following dimension
and angle conditions of Wang--Zhang~\cite{WZ26}.
\begin{equation}\label{stable-cone-range}
 \begin{cases}
 0<\theta<\pi, & n=3,4,\\
 73.336^\circ<\theta<106.664^\circ, & n=5,\\
 85.420^\circ<\theta<94.580^\circ, & n=6.
 \end{cases}
\end{equation}

\begin{samepage}
\begin{theo}\label{index-one}
Let \(M^n\subset\mathbb R^{n+1}_+\) be a connected, complete, two-sided,
properly embedded capillary minimal hypersurface with nonempty boundary,
contact angle
\(\theta\), extrinsic Euclidean volume growth, and \(\Ind_\theta(M)=1\).
Suppose that \((n,\theta)\) satisfies \eqref{stable-cone-range}.
Then, up to dilations and rigid motions,
\begin{enumerate}
\item if \(\partial M\) is compact, then \(M\) is a truncation of the
\(n\)-dimensional catenoid;
\item if \(\partial M\) is noncompact, then \(M\) is the
\(n\)-dimensional capillary hairpin surface.
\end{enumerate}
\end{theo}
\end{samepage}

For the history of the index-one classification of minimal hypersurfaces
without boundary,
we refer to Chodosh--Gianocca~\cite{CG26} and the references therein.
Tam--Zhou~\cite[Theorem 2.1]{TZ09} proved that the \(n\)-dimensional
catenoid in \(\mathbb R^{n+1}\) has Morse index one for every \(n\ge3\).
For \(n\ge3\) and \(\theta\ne\frac{\pi}{2}\), the capillary hairpin models
have the graph description
\eqref{hairpin-graph}--\eqref{hairpin-graph-equation} in
Appendix~\ref{sec:hairpin-construction}.
In dimension two, separation of variables  gives
\(F''=F\), \(F(0)=1\), and \(F'(0)=0\).
The solution \(F(s)=\cosh s\) yields Chodosh's
parametrization~\cite{ChodoshWeb}:
\begin{equation}\label{hairpin-param}
 \begin{split}
 X_c(s,t)&=\bigl(s+c\sinh s\cos t,
       ct+\cosh s\sin t,\sqrt{1-c^2}\cosh s\cos t\bigr),\\
 &\hspace{12mm}s\in\mathbb R,\quad |t|\le\frac{\pi}{2},\quad 0\le c<1.
 \end{split}
\end{equation}
Here \(c=|\cos\theta|\).
A direct calculation gives
$\Ind_\theta(X_c)=1.$
Appendix~\ref{sec:index-calculations} gives the construction and Morse
index calculation for the higher-dimensional capillary hairpins,
for every \(n\ge3\) and every contact angle.

\begin{figure}[H]
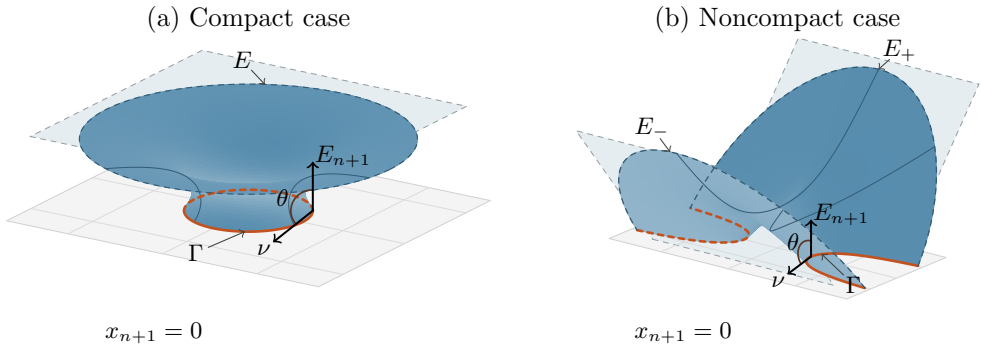

\centering

\caption{(a) Catenoid truncation. (b) Capillary hairpin hypersurface.}
\label{fig:index-one}
\end{figure}

A main ingredient in the proof of Theorem~\ref{index-one} is the following
finite-index structure theorem. It extends regularity at infinity for
minimal hypersurfaces without boundary~\cite[Appendix A]{CG26},
\cite{Schoen83} to capillary hypersurfaces.

\begin{theo}\label{main}
Let \(M^n\subset\mathbb R^{n+1}_+\) be a connected, complete, two-sided,
properly embedded capillary minimal hypersurface with nonempty boundary,
contact angle
\(\theta\), extrinsic Euclidean volume growth, and \(\Ind_\theta(M)<\infty\).
Suppose that
\((n,\theta)\) satisfies \eqref{stable-cone-range}.
Then \(M\) has finite total curvature and finitely many ends:
\[
 \int_M|h|^n<\infty,\qquad
 M\setminus K=E_1\sqcup\cdots\sqcup E_N
\]
for a compact set \(K\subset M\). Each end is a single-valued graph with a
unique tangent cone of multiplicity one.

In suitable coordinates, an interior end is a graph over a horizontal
affine hyperplane and satisfies
\begin{equation}\label{interior-exp}
 x_{n+1}=b+a|x'|^{2-n}+\frac{d\cdot x'}{|x'|^n}
                 +O_k(|x'|^{-n}),\qquad x'=(x_1,\ldots,x_n),\quad d\in\mathbb R^n.
\end{equation}
A boundary end is a capillary graph satisfying
\begin{equation}\label{boundary-exp}
 \begin{gathered}
 x_n=-\cot\theta\,x_{n+1}+b+a\rho^{2-n}
     +\frac{d\cdot x'}{\rho^n}
     +
 \begin{cases}
  O_k(\rho^{-3}\log\rho), & n=3,\\
  O_k(\rho^{-n}), & n\ge4,
 \end{cases}\\
 \rho^2=|x'|^2+\frac{x_{n+1}^2}{\sin^2\theta},
 \qquad x'=(x_1,\ldots,x_{n-1}),\quad d\in\mathbb R^{n-1}.
 \end{gathered}
\end{equation}
Moreover, \(|\nabla^jh|=O(r^{-n-j})\), where \(r=|X|\).
Here \(f=O_k(r^\beta)\) means \(|D^jf|\le C_jr^{\beta-j}\) for
\(0\le j\le k\).
\end{theo}

\medskip
\noindent\textbf{Strategy and proof outline.}
For Theorem~\ref{index-one-surface}, finite index gives finite total
curvature and finite conformal type~\cite[Theorem 1.4]{HS23}.
Schwarz reflection of the Gauss map and positivity of the height give
regular boundary ends. The harmonic one-form method of
Ros~\cite{Ros06} and Chodosh--M\'aximo~\cite{CM18}, adapted to the
capillary boundary condition, gives the topological bound in
Lemma~\ref{surface-index-bound}.
For compact boundary, the L\'opez--Ros deformation~\cite{LR91,Choe93}
and reflection reduce the classification to Schoen's
theorem~\cite[Theorem 3]{Schoen83}.
For noncompact boundary, the index bound and monotonicity determine
the topology; the height function and Gauss map then determine the
Weierstrass data and give \eqref{hairpin-param}.

For Theorem~\ref{index-one}, we use the end expansions in
Theorem~\ref{main}.
For compact boundary, we combine Chodosh--Gianocca's harmonic
one-form method~\cite[Sections 5--8]{CG26}(based on
\cite{CM16,CM18,Li17,Sav10}) ,with the capillary boundary identities
to obtain one end and one boundary component.
Moving planes then gives rotational symmetry and the catenoid.
For noncompact boundary, embeddedness and monotonicity give two
boundary ends, and moving planes gives the symmetries in
Proposition~\ref{two-ends}.
After fixing the center and scale, we invert the linearized problem
and apply the implicit function theorem to vary the contact angle.
Appendix~\ref{sec:hairpin-construction} uses this construction and
uniform estimates to obtain hairpins of index one for every contact
angle. Under \eqref{stable-cone-range}, compactness and uniqueness of
the half-catenoid at \(\theta=\frac{\pi}{2}\) give uniqueness for
each allowed contact angle.

 For Theorem~\ref{main}, we use stability outside a compact set
and a standard blow-down argument.
Wang--Zhang's curvature estimates and stable-cone classification
\cite[Corollary 1.10 and Theorem 1.16]{WZ26}
then give \(r|h|\to0\).
Harmonic approximation gives \(|h|=O(r^{-1-\alpha})\)
for some \(\alpha>1\), hence finite total curvature.
Integrating this curvature bound gives a limiting unit normal
on each end. The single-valued graph representation then gives
a unique tangent cone at infinity with multiplicity one.
This extends the asymptotic analysis in
\cite[Section 2]{Schoen83} and \cite[Theorem A.4]{CG26}
to the capillary setting.
The minimal graph equation and the contact-angle boundary condition
give the end expansions and their derivative estimates.

\medskip
\noindent\textbf{Organization of this paper.}
In Section~\ref{sec:preliminaries}, we collect the geometric conventions,
boundary identities, and Sobolev inequalities for capillary hypersurfaces.
Section~\ref{sec:classification-two} proves
Theorem~\ref{index-one-surface}.
Section~\ref{sec:structure} establishes the finite total curvature and
end expansions in Theorem~\ref{main}.
Section~\ref{sec:classification-three} proves Theorem~\ref{index-one}.
Appendix~\ref{sec:hairpin-construction} gives the higher-dimensional
hairpin construction and Morse index calculation.
Appendix~\ref{sec:surface-ends-proof} proves regularity of the ends in
dimension two.

\medskip
\noindent\textbf{Acknowledgments.}
The author is grateful to Professors~Guofang Wang and~Chao Xia
for their constant support and interest in this work.

\medskip
\noindent\textbf{Disclosure on AI assistance.}
The main ideas of this paper are due to the author.
The author initially computed the Morse index of the two-dimensional
capillary hairpin surface with assistance from ChatGPT 5.6.
This calculation motivated the author to study the classification problem.
In the construction of higher-dimensional capillary hairpin
hypersurfaces, ChatGPT 6.0 also assisted with the proof of compactness
in weighted function spaces (Proposition~\ref{properness}).
Figure~\ref{fig:index-one} was drawn by ChatGPT 6.0, which also assisted
with other calculations and PDE estimates.
The author takes full responsibility for the content of this paper.

\section{Preliminaries}\label{sec:preliminaries}

Let \(M^n\), \(n\ge2\), be a smooth manifold with nonempty boundary, and let
\[
 X:(M,\partial M)\longrightarrow(\mathbb R^{n+1}_+,\partial\mathbb R^{n+1}_+)
\]
be a two-sided proper embedding. We identify \(M\) with its image and
assume that it is capillary minimal with contact angle \(\theta\in(0,\pi)\).

Let \(\nu\) be a global unit normal and \(\mu\) the outward unit conormal along
\(\partial M\). We denote the standard coordinate vectors by
\(E_1,\ldots,E_{n+1}\). Let \(\bar\nu\) be the unit normal to \(\partial M\)
within the hyperplane \(\{x_{n+1}=0\}\).
Choose its orientation so that
\begin{equation*}
 \mu=-\sin\theta\,E_{n+1}+\cos\theta\,\bar\nu,\qquad
 \nu=\cos\theta\,E_{n+1}+\sin\theta\,\bar\nu .
\end{equation*}
Let \(\bar\nabla,\nabla\) be the Euclidean and induced connections.
The second fundamental form is
\[
 h(U,V)=\langle\bar\nabla_U\nu,V\rangle.
\]
We write \(\nu_i=\langle\nu,E_i\rangle\).

For \(f\in C_c^\infty(M)\) and \(g\in C^\infty(M)\),
integration by parts gives
\[
 Q(f,g)=-\int_M f(\Delta g+|h|^2g)\,dA
       +\int_{\partial M}f(\partial_\mu g-qg)\,ds.
\]
Along \(\partial M\), the angle condition and minimality give
\begin{equation}\label{boundary}
 h(\mu,T)=0\quad(T\in T\partial M),\qquad
 \operatorname{div}_{\partial M}\mu=-q.
\end{equation}
For \(0\le w\in C_c^\infty(M)\), the trace estimate
\cite[Lemma 2.2]{WZ26} gives
\begin{equation}\label{height-trace}
 \sin\theta\int_{\partial M}w\le\int_M|dw|.
\end{equation}
For \(w,f\in C_c^\infty(M)\), the capillary Sobolev inequality
\cite[Lemma 2.3]{WZ26} and \eqref{height-trace}, together with
H\"older's inequality, give
\begin{equation}\label{Sob}
 \begin{aligned}
 \|w\|_{L^{\frac{n}{n-1}}(M)}
 &\le C_{n,\theta}\|dw\|_{L^1(M)},\\
 \|f\|_{L^{\frac{2n}{n-2}}(M)}
 +\|f\|_{L^{\frac{2(n-1)}{n-2}}(\partial M)}
 &\le C_{n,\theta}\|df\|_{L^2(M)}.
 \end{aligned}
\end{equation}
The second inequality requires \(n\ge3\).
Equation~\eqref{Sob} also gives
\[
 \Vol(B_M(p,R))\ge a_{n,\theta}R^n.
\]

Let \(\Omega\) be the component of \(\mathbb R^{n+1}_+\setminus M\)
for which \(\nu\) is the outward unit normal along \(M\), and let
\(D=\overline\Omega\cap\partial\mathbb R^{n+1}_+\).
The outward normal of \(\partial D\) within the hyperplane
\(\{x_{n+1}=0\}\) is \(\bar\nu\).
The capillary monotonicity formula \cite[Lemma 4.1]{WZ26} gives, for
\(p\in\partial\mathbb R^{n+1}_+\) and \(0<\rho<R\),
\begin{equation}\label{mono}
 \begin{split}
 &\frac{\Vol(M\cap B_R(p))-\cos\theta\,\Vol(D\cap B_R(p))}{R^n}\\
 &\quad-\frac{\Vol(M\cap B_\rho(p))-\cos\theta\,\Vol(D\cap B_\rho(p))}{\rho^n}
 =\int_{M\cap(B_R(p)\setminus B_\rho(p))}
       \frac{|(X-p)^\perp|^2}{|X-p|^{n+2}}\,dA.
 \end{split}
\end{equation}

\section{The two-dimensional classification}\label{sec:classification-two}

\begin{lemm}\label{surface-ends}
Let \(M^2\subset\mathbb R^3_+\) be a connected, complete, two-sided,
properly embedded capillary minimal surface with nonempty boundary and
finite index. Then \(M\) has finite total curvature and is conformally equivalent to
\[
 \overline M\setminus\{p_1,\ldots,p_N\},
\]
where \(\overline M\) is a compact Riemann surface with boundary
and \(p_1,\ldots,p_N\) are distinct points of \(\overline M\). Every end is a
single regular graph with a unique limiting plane or capillary half-plane.
If \(e_i\) and \(e_b\) denote the numbers of interior and boundary ends,
respectively, then
\[
 \lim_{R\to\infty}\frac{\operatorname{Area}(M\cap B_R)}{\pi R^2}
       =e_i+\frac{e_b}{2}.
\]
\end{lemm}
The proof of Lemma~\ref{surface-ends} is given in Appendix~\ref{sec:surface-ends-proof}.

For \(M\) as in Lemma~\ref{surface-ends}, use the Weierstrass data
\((G,\omega)\) of Appendix~\ref{sec:surface-ends-proof}, with
\(dX=\operatorname{Re}\Phi\) and \(|G|=k>0\) on \(\partial M\).
Let \(g,b\) be the genus and number of boundary circles of the conformal
compactification.

\begin{lemm}\label{surface-index-bound}
If \(M\) is nonplanar, then
\[
 3\Ind_\theta(M)\ge2g+b+4e_i+2e_b-3.
\]
\end{lemm}
\begin{proof}
Glue \(\overline M\) to its complex conjugate along the boundary.
Consider meromorphic one-forms \(\omega\) on this closed surface,
with poles of order at most two at the added points and
\(\omega(T)\in\mathbb R\) for every boundary tangent vector \(T\).
Let \(\mathcal H\) be the real vector space of one-forms
\(\alpha=\operatorname{Re}(\omega|_M)\).
The real dimension of \(\mathcal H\) equals the complex dimension
of the meromorphic one-forms without the restriction
\(\omega(T)\in\mathbb R\).
Riemann--Roch therefore gives
\[
 \dim_{\mathbb R}\mathcal H=2g+b+4e_i+2e_b-2,\qquad
 d\alpha=\delta\alpha=0,\qquad
 \alpha(\mu)=0\quad\text{on }\partial M.
\]

For \(\xi=\alpha^\sharp\), write \(\xi_j=\langle\xi,E_j\rangle\).
These functions extend smoothly to every added point. In fact, if
\(\Phi=\phi(z)\,dz\) and \(\alpha=\operatorname{Re}(a(z)\,dz)\), then
\[
 \xi=\frac{2\operatorname{Re}(a\bar\phi)}{|\phi|^2}.
\]
Both \(z^2a\) and \(z^2\phi\) are holomorphic at \(z=0\), with
\((z^2\phi)(0)\ne0\) by Lemma~\ref{surface-ends}.

Let \(T\) be a unit boundary tangent. Since
\(\langle\nabla_T\mu,T\rangle=-q\), the Bochner and boundary identities are
\[
 \frac12\Delta|\xi|^2=|\nabla\xi|^2+K|\xi|^2,
 \qquad \langle\nabla_\mu\xi,\xi\rangle=q|\xi|^2.
\]
Since \(|h(\xi,\cdot)|^2=-K|\xi|^2\), integration gives
\[
 \sum_{j=1}^3 Q(\eta\xi_j)
       =\int_M|\nabla\eta|^2|\xi|^2
 \qquad(\eta\in C_c^\infty(M)).
\]
Choose \(\eta_R\in C_c^\infty(M)\) satisfying
\[
 \begin{gathered}
 0\le\eta_R\le1,\qquad
 \eta_R=1\ \text{on }M\cap B_R,\qquad
\operatorname{supp}\eta_R\subset B_{R^2},\\
 \int_M|\nabla\eta_R|^2\,dA
 \le\frac{C}{\log R}\longrightarrow0.
 \end{gathered}
\]
Thus
\[
 \sum_{j=1}^3Q(\xi_j)
 =\lim_{R\to\infty}\int_M|\nabla\eta_R|^2|\xi|^2\,dA=0.
\]

More precisely, for a smooth conformal background metric on \(\overline M\),
\[
 Q(f)=\int_{\overline M}(|\nabla_0 f|^2-Vf^2)\,dA_0
             -\int_{\partial\overline M}\beta f^2\,ds_0,
 \quad V\,dA_0=|h|^2dA,\quad \beta\,ds_0=q\,ds.
\]
The coefficients are smooth at the added points. Along the boundary,
$
 \beta=-\frac{\cos\theta}{\sin^2\theta}\partial_{\mu_0}\nu_3,
$
where \(\mu_0\) is the outward conormal for the background metric.
The cutoff approximation gives
\[
 \overline{C_c^\infty(M)}^{\,H^1(\overline M)}
 =H^1(\overline M),\qquad
 \Ind_\theta(M)=\Ind(Q;H^1(\overline M)).
\]
Let \(\varphi_1,\ldots,\varphi_{\Ind_\theta(M)}\) be an
\(L^2(\overline M,dA_0)\)-orthonormal basis of negative eigenfunctions.
Define the linear map
\[
 \mathcal L:\mathcal H\longrightarrow
 \mathbb R^{3\Ind_\theta(M)},\qquad
 \mathcal L(\alpha)=
 \left(\int_{\overline M}\xi_j\varphi_\ell\,dA_0\right)_{
 \substack{1\le j\le3\\1\le\ell\le\Ind_\theta(M)}}.
\]
Set \(\mathcal H_*=\ker\mathcal L\). Then
\[
 \dim\mathcal H_*
 =\dim\mathcal H-\operatorname{rank}\mathcal L
 \ge\dim\mathcal H-3\Ind_\theta(M).
\]
For \(\alpha\in\mathcal H_*\),
\[
 Q(\xi_j)\ge0,\qquad
 \sum_{j=1}^3Q(\xi_j)=0
 \quad\Longrightarrow\quad Q(\xi_j)=0.
\]
Thus, for \(j=1,2,3\),
\[
 (\Delta+|h|^2)\xi_j=0\quad\text{in }M,\qquad
 (\partial_\mu-q)\xi_j=0\quad\text{on }\partial M.
\]

By \cite[Lemma 1]{Ros06},
\[
 \begin{gathered}
 (\Delta+|h|^2)\xi=-2\langle h,\nabla\xi\rangle\nu=0,\quad
 \alpha=*d\langle a,X\rangle,\quad a\in\mathbb R^3,
 \end{gathered}
\]
where \(*\) is the intrinsic Hodge star. On \(\partial M\),
\[
 \alpha(\mu)=0
 \quad\Longrightarrow\quad
 \langle a,T\rangle=0.
\]
If \((a_1,a_2)\ne(0,0)\), each boundary component \(\Gamma\)
satisfies
\[
 \Gamma\subset\{x_3=0,\ a_1x_1+a_2x_2=c_\Gamma\}.
\]
The contact-angle condition gives
\(\bar\nabla_T\nu=0\), hence \(G|_\Gamma=\mathrm{const}\).
Schwarz reflection using \(|G|=k\) and the identity theorem imply
\(G\equiv\mathrm{const}\) on \(M\), contradicting nonplanarity.
Thus
$
 \mathcal H_*\subset\operatorname{span}\{*dx_3\},
 \dim\mathcal H-3\Ind_\theta(M)
 \le\dim\mathcal H_*\le1.
$
\end{proof}

\begin{lemm}\label{boundary-curvature-index}
If \(M\) is nonplanar and
$
 \int_{\partial M}h(\mu,\mu)\,ds=0,
$
then \(\Ind_\theta(M)\ge2\).
\end{lemm}
\begin{proof}
The integrals converge by Lemma~\ref{surface-ends}. Since
\[
 (\Delta+|h|^2)\nu_3=0,
 \qquad (\partial_\mu-q)\nu_3=-\frac{h(\mu,\mu)}{\sin\theta},
\]
Green's formula for the compactified form and the hypothesis give
\[
 Q(1)=-\int_M|h|^2\,dA<0,
 \qquad Q(1,\nu_3)=Q(\nu_3)=0.
\]
Suppose that \(\Ind_\theta(M)\le1\).
Since \(Q(1)<0\), for \(f\in H^1(\overline M)\) with \(Q(f,1)=0\),
\[
 0\le Q(\nu_3+tf)
 =2tQ(\nu_3,f)+t^2Q(f)\qquad(t\in\mathbb R).
\]
Hence \(Q(\nu_3,f)=0\). Since \(Q(\nu_3,1)=0\), for arbitrary \(f\),
\[
 Q(\nu_3,f)
 =Q\left(\nu_3,f-\frac{Q(f,1)}{Q(1)}\cdot1\right)=0.
\]
The Robin identity gives
\[
 0=(\partial_\mu-q)\nu_3
 =-\frac{h(\mu,\mu)}{\sin\theta}.
\]
By \eqref{boundary} and minimality,
\[
 h(\mu,\mu)=h(\mu,T)=h(T,T)=0
 \quad\Longrightarrow\quad
 \bar\nabla_T\nu=0.
\]
The Schwarz reflection argument in the proof of
Lemma~\ref{surface-index-bound} gives
\(G\equiv\mathrm{const}\), contradicting nonplanarity.
\end{proof}

\begin{lemm}\label{surface-periods}
If \(H_1(M;\mathbb Z)\) is
generated by compact components of \(\partial M\), then \(G\mapsto \frac{G}{k}\)
defines a complete, single-valued free-boundary minimal immersion with
the same height function. Reflecting this free-boundary immersion in \(\{x_3=0\}\)
gives a complete minimal immersion with regular embedded ends.
\end{lemm}
\begin{proof}
For a compact boundary component \(\Gamma\), the vanishing real horizontal
periods and \(|G|=k\) give
\[
 \int_\Gamma\frac\omega G
 =\overline{\int_\Gamma G\omega}
 =-k^2\int_\Gamma\frac\omega G.
\]
The identity above and the hypothesis on \(H_1(M;\mathbb Z)\) give
\[
 \int_\gamma G\omega=\int_\gamma\frac{\omega}{G}=0
 \quad\text{for every closed curve }\gamma\subset M.
\]
Thus the data \(\widetilde G=\frac{G}{k}\) and
\(\widetilde\omega=\omega\) define a single-valued free-boundary
immersion with the same height.
By \cite[Theorems 1 and 2]{Choe93}, it is complete and extends
by reflection.
The deformation satisfies \(\widetilde\Phi=T_k\Phi\), with
\(T_k\in O(3,\mathbb C)\), so its double-pole coefficient is
\(\widetilde A=T_kA\ne0\).
At a reflected end,
\[
 \Phi=(\frac{A}{z^2}+\frac{B}{z}+O(1))\,dz,\qquad
 A\cdot A=0,\quad A\ne0,\quad B\in\mathbb R^3,\quad A\cdot B=0.
\]
Since \(B\in\mathbb R^3\) and \(A\cdot B=0\),
\(B\perp\operatorname{Re}A\) and \(B\perp\operatorname{Im}A\).
The end is the graph of a single-valued function \(u\) over
\(\operatorname{span}_{\mathbb R}
\{\operatorname{Re}A,\operatorname{Im}A\}\), with
\(u=O(\log r)\) and \(|Du|=O(r^{-1})\).
\end{proof}

\begin{proof}[Proof of Theorem~\ref{index-one-surface}]
Index one implies nonplanarity. Lemma~\ref{surface-index-bound} gives
\[
 2g+b+4e_i+2e_b\le6.
\]
Write \(\partial\overline M=\Gamma_1\sqcup\cdots\sqcup\Gamma_b\).

If \(\partial M\) is compact, then
\[
 e_i=1,\qquad g=0,\qquad b\in\{1,2\},\qquad
 H_1(M;\mathbb Z)=\sum_{\alpha=1}^{b}\mathbb Z[\Gamma_\alpha].
\]
By Lemma~\ref{surface-periods} and
\cite[Theorem 3]{Schoen83}, the reflected immersion
\(\widehat M\) is a catenoid. Thus
\[
 0=\operatorname{genus}(\widehat M)=2g+b-1
 \quad\Longrightarrow\quad b=1.
\]
Since \(\widehat M\cap\{x_3=0\}\) is compact, the reflection
plane contains its circle of minimum radius.
The inverse deformation gives a horizontal catenoid truncation.

If \(\partial M\) is noncompact, Lemma~\ref{surface-index-bound}
gives \(e_i=0\).
The case \(e_b=1\) would give a half-plane by \eqref{mono},
contradicting index one. Hence
\[
 e_i=0,\qquad e_b=2,\qquad g=0,\qquad b\in\{1,2\}.
\]
Suppose \(b=2\), and let \(p_1,p_2\) be the two deleted boundary
points. If \(p_1,p_2\in\Gamma_1\), then
\[
 H_1(M;\mathbb Z)=\mathbb Z[\Gamma_2].
\]
The same application of Lemma~\ref{surface-periods} and
\cite[Theorem 3]{Schoen83} gives a catenoid, and
\[
 0=\operatorname{genus}(\widehat M)=2g+b-1=1,
\]
a contradiction.

If \(p_1,p_2\) lie on different components of
\(\partial\overline M\), each boundary component \(\Gamma\)
has an arclength parametrization with
\[
 \gamma(s)=se+O(\log|s|),\qquad
 \gamma'(s)=e+O(|s|^{-1})\qquad(s\to\pm\infty),
\]
where \(e\) is a horizontal unit vector.
Choose \(\sigma_R\subset\partial B_R\cap\{x_3=0\}\) so that
\((\Gamma\cap\overline B_R)\cup\sigma_R\) is a simple closed curve.
Let \(\alpha_R,\beta_R\) be its two exterior angles. Then
\[
 \int_{\sigma_R}\kappa\,ds\longrightarrow\pi,\qquad
 \alpha_R,\beta_R\longrightarrow\frac{\pi}{2}.
\]
Gauss--Bonnet gives
\[
 \int_\Gamma\kappa\,ds
 =\lim_{R\to\infty}
 \left(2\pi-\int_{\sigma_R}\kappa\,ds-\alpha_R-\beta_R\right)
 =0.
\]
By minimality and the contact-angle condition,
\[
 h(\mu,\mu)=-h(T,T)
 =-\sin\theta\,\langle\bar\nabla_T\bar\nu,T\rangle.
\]
Consequently,
\[
 \int_{\partial M}h(\mu,\mu)\,ds
 =\sum_\Gamma\int_\Gamma h(\mu,\mu)\,ds=0,
\]
contradicting Lemma~\ref{boundary-curvature-index}.
Hence \(b=1\).

Identify the conformal compactification with \(\overline{\mathbb D}\),
and let \(p_1,p_2\in\partial\mathbb D\) denote the two ends. Then
\[
 \begin{aligned}
 x_3(z)&=\sum_{j=1}^2a_j
       \frac{1-|z|^2}{|z-p_j|^2},\qquad a_j>0,\\
 \omega=2\partial x_3
       &=2\sum_{j=1}^2\frac{a_jp_j}{(p_j-z)^2}\,dz.
 \end{aligned}
\]
On the doubled sphere, \(\omega\) has no other poles, so
\[
 \operatorname{ord}_{p_1}\omega
 =\operatorname{ord}_{p_2}\omega=-2,\qquad
 \sum_{\omega(q)=0}\operatorname{ord}_q\omega=4-2=2.
\]
Hopf's lemma gives \(\omega\ne0\) on
\(\partial\mathbb D\setminus\{p_1,p_2\}\).
Schwarz reflection preserves the order of each zero; hence
\(\sum_{q\in\mathbb D,\ \omega(q)=0}\operatorname{ord}_q\omega=1\).
Thus \(\omega\) has exactly one zero in \(\mathbb D\), of order one.

Choose a disk coordinate with the unique zero of \(\omega\) at \(0\).
Since this zero is simple and the immersion is regular,
\(\operatorname{ord}_0G=\pm1\).
With a suitable choice of horizontal coordinates, the boundary
condition \(|G|=k\) gives \(G=kz\) or \(G=\frac{k}{z}\).
Let \(a_1,a_2>0\) be the coefficients in the expression for \(x_3\).
Then
\[
 \omega(0)=2\left(\frac{a_1}{p_1}+\frac{a_2}{p_2}\right)dz=0.
\]
Since \(|p_1|=|p_2|=1\), we obtain \(a_1=a_2\) and \(p_2=-p_1\).
Using \(\overline{p_1}z\) as the disk coordinate, we may take
\(p_1=1\) and \(p_2=-1\).
With a suitable choice of scale and horizontal coordinates,
integrating \(dX=\operatorname{Re}\Phi\) in
\(z=\tanh\left(\frac{s+it}{2}\right)\) gives
the parametrization~\eqref{hairpin-param}, extended to \(-1<c<1\).

\end{proof}

\section{Finite total curvature and asymptotic structure}\label{sec:structure}

\begin{proof}[Proof of Theorem~\ref{main}]
Let \(r=|X|\). Finite index implies stability outside a compact set.
The curvature estimates of Wang--Zhang
\cite[Corollary 1.10 and Remark 1.11]{WZ26}, together with local elliptic
estimates for capillary graphs, give
\[
 \sup_{M\cap(B_{2R}\setminus B_R)}
 \bigl(R|h|+R^2|\nabla h|\bigr)\le C
\]
for all sufficiently large \(R\).

Choose the orientation so that \(\cos\theta\le0\). By \eqref{mono}
with \(p=0\), the capillary area ratio has a finite limit, since
\(\Vol(D\cap B_R)\le\omega_nR^n\).

Let \(R_j\to\infty\), \(M_j=R_j^{-1}M\), and
\[
 A_{s,t}=\{s<|x|<t\},\qquad
 \mathcal A(R)=\frac{\Vol(M\cap B_R)-\cos\theta\,\Vol(D\cap B_R)}{R^n}.
\]
For fixed \(0<s<t<\infty\) and large \(j\),
\[
 Q_{M_j}(f)\ge0\quad(f\in C_c^\infty(M_j\cap A_{s,t})),\qquad
 \Vol(M_j\cap B_t)\le Ct^n.
\]
By \cite[Theorem 6.8]{WZ26} and the monotonicity
formula~\eqref{mono}, a subsequence converges smoothly with finite
multiplicity on compact annuli to a proper, two-sided stable
capillary minimal cone. Let \(C\) be a connected component of this cone.

If \(C\subset\{x_{n+1}=0\}\), then \(h_C=0\).
Otherwise, when \(\partial C\ne\varnothing\),
\eqref{stable-cone-range} and
\cite[Theorem 1.16]{WZ26} give \(h_C=0\).
When \(\partial C=\varnothing\), set
\(L_C=C\cap\mathbb S^n\) and \(\xi=x_{n+1}|_{L_C}\). Then
\[
 \Delta_{L_C}\xi=-(n-1)\xi,\qquad \xi\ge0,\qquad
 0=\int_{L_C}\Delta_{L_C}\xi
   =-(n-1)\int_{L_C}\xi.
\]
Hence \(\xi\equiv0\), so \(C\subset\{x_{n+1}=0\}\)
and \(h_C=0\).

If \(r|h|\not\to0\), then there are points \(p_j\), with
\(r_j=|p_j|\to\infty\) and \(r_j|h(p_j)|\ge\epsilon_0>0\).
Rescaling by \(r_j^{-1}\) gives a nonflat limiting component, a contradiction.

To prove that \(M\) has finitely many ends, we apply Morse theory
to the function \(r^2/2\). Since \(r|h|\to0\), for sufficiently large \(r\),
\[
 \begin{aligned}
 \operatorname{Hess}_M\left(\frac{r^2}{2}\right)
 =g-(X\cdot\nu)h
 \ge\frac{3}{4}g, 
 \operatorname{Hess}_{\partial M}\left(\frac{r^2}{2}\right)
 =g^{\partial M}
 -\frac{(X\cdot\bar\nu)h|_{T\partial M}}{\sin\theta}
 \ge\frac{3}{4}g^{\partial M}.
 \end{aligned}
\]
Choose \(0\le\chi_R\le1\) with \(\chi_R=1\) on \(B_R\),
\(\operatorname{supp}\chi_R\subset B_{2R}\), and
\(|d\chi_R|\le\frac{C}{R}\).
By \eqref{height-trace},
\[
 \sin\theta\,|\partial M\cap B_R|
 \le\int_M|d\chi_R|
 \le\frac{C}{R}\Vol(M\cap B_{2R})
 \le CR^{n-1}.
\]

Choose a regular radius \(R_0\) beyond which both Hessians are positive.
By properness and Morse theory, each connected component
\(\Sigma\subset\partial M\) satisfies
\[
 \operatorname{Crit}\left(\left.\frac{r^2}{2}\right|_\Sigma\right)\setminus B_{R_0}
 =\begin{cases}
  \varnothing, & \Sigma\cap B_{R_0}\ne\varnothing,\\
  \{p_\Sigma\}, & \Sigma\cap B_{R_0}=\varnothing.
 \end{cases}
\]
In the second case, \(p_\Sigma\) is the unique minimum and \(\Sigma\)
is unbounded. Write \(r_0=|p_\Sigma|\).
For \(R>4r_0\), choose \(p\in\Sigma\) with \(|p|=\frac{R}{2}\). On
\(B_{\frac{R}{8}}(p)\), tangential differentiation of the angle condition gives
\[
 |H_\Sigma|\le\frac{(n-1)|h|}{\sin\theta}\le \frac{C_{n,\theta}}{R}.
\]
The monotonicity inequality for bounded mean curvature, applied within
the hyperplane \(\{x_{n+1}=0\}\), therefore yields
\[
 |\Sigma\cap B_R|
 \ge|\Sigma\cap B_{\frac{R}{8}}(p)|
 \ge e^{-\frac{C_{n,\theta}}{8}}\omega_{n-1}(\frac{R}{8})^{n-1}
 \ge\gamma_{n,\theta}R^{n-1}.
\]
Here \(\gamma_{n,\theta}>0\).
For any \(m\) such components \(\Sigma_1,\ldots,\Sigma_m\), choose
\(R>4\max_i\min_{\Sigma_i}r\). Then
\[
 m\gamma_{n,\theta}R^{n-1}\le|\partial M\cap B_R|
                             \le C_{n,\theta}R^{n-1}.
\]
The area estimates give
\[
 \#\{\Sigma\in\pi_0(\partial M):\Sigma\cap B_{R_0}=\varnothing\}
 \le \frac{C_{n,\theta}}{\gamma_{n,\theta}}.
\]
Enlarge \(R_1\) so that all critical points of
\(\left.\frac{r^2}{2}\right|_{\partial M}\) lie in \(B_{R_1}\).
An interior critical point outside \(B_{R_1}\) would be a
nondegenerate local minimum.
By \cite[Theorem 3.2]{Mil63}, it would create a component of
\(M\cap B_R\) that remains disjoint from \(M\cap B_{R_1}\)
as \(R\) increases, contradicting connectedness of \(M\).
Thus \(r\) has no critical points on either \(M\) or \(\partial M\)
outside \(B_{R_1}\).
Properness and the boundary-preserving flow
\cite[Lemma 2-10]{BNR16} give
\[
 \begin{gathered}
 M\setminus B_{R_*}
 \cong(M\cap\partial B_{R_*})\times[R_*,\infty),\\
 M\setminus B_{R_*}=E_1\sqcup\cdots\sqcup E_N,\qquad N<\infty.
 \end{gathered}
\]
Here \(R_*>R_1\), and finiteness of \(N\) follows from compactness
of \(M\cap\partial B_{R_*}\).
Each end has connected level sets and annular regions.
The same flow restricts to \(\partial M\).

Fix \(B>2\). A standard covering argument, using the intrinsic volume
lower bound and Euclidean volume growth, covers
\(E\cap\{2R\le r\le BR\}\) by at most \(C_B\) intrinsic balls of
radius \(\frac{R}{4}\), centred in this annulus. Connectedness then joins any
two points \(p,q\) by a path \(\gamma\subset E\cap\{R<r<2BR\}\)
with \(\ell(\gamma)\le C_BR\). Hence
\[
 |\nu(p)-\nu(q)|\le\int_\gamma|h|\,ds
 \le C_B\sup_{E\cap\{r\ge R\}}r|h|.
\]
Taking the supremum gives
\begin{equation}\label{normal-oscillation}
 \operatorname{osc}_{E\cap\{2R\le r\le BR\}}\nu
 \le C_B\sup_{E\cap\{r\ge R\}}r|h|.
\end{equation}
Integrating along a path from radius \(R\) to radius \(\eta R\), and taking
the inner product with its initial normal, we obtain
\[
 |X\cdot\nu|\le\eta R+o_\eta(R),\qquad X\cdot\nu=o(r).
\]
The second assertion follows by letting \(R\to\infty\), then \(\eta\to0\).

For a boundary end \(E\), choose \(p_R\in E\cap\partial M\) with \(|p_R|=3R\) and set
\(\nu_R=\sin\theta\,\bar\nu_R+\cos\theta\,E_{n+1}\). Projection onto
\(\nu_R^\perp\) along \(\bar\nu_R\) is given by
\[
 \Pi_RX=X-\frac{X\cdot\nu_R}{\sin\theta}\bar\nu_R
       \in\nu_R^\perp.
\]
The map \(\Pi_R\) preserves \(x_{n+1}\). On fixed enlarged annuli,
\[
 \nu\cdot\bar\nu_R=\sin\theta+o(1),\quad
 X\cdot\nu_R=o(r),\quad |\Pi_RX|=r(1+o(1)).
\]
On a connected level set \(\Sigma_{r_0}\), define
\[
 \Psi_R(X)=\frac{\Pi_RX}{|\Pi_RX|}\in H_R,\qquad
 H_R=\nu_R^\perp\cap S^n\cap\{x_{n+1}\ge0\}.
\]
A vector in the kernel of \(D\Psi_R\) has
\[
 V=\lambda\left(X-\frac{X\cdot\nu}{\bar\nu_R\cdot\nu}\bar\nu_R\right),
 \quad 0=X\cdot V
 =\lambda\left(r_0^2-
 \frac{(X\cdot\nu)(X\cdot\bar\nu_R)}{\bar\nu_R\cdot\nu}\right).
\]
The last factor is \(r_0^2(1+o(1))\), so \(\lambda=0\) and the map is
locally invertible.
Along \(\partial\Sigma_{r_0}\subset\partial M\),
\[
 |\nabla_{\Sigma_{r_0}}x_{n+1}|^2
 =\sin^2\theta-
 \frac{\cos^2\theta\,(X\cdot\nu)^2}{r_0^2-(X\cdot\nu)^2}
 =\sin^2\theta+o(1)>0.
\]
Since \(x_{n+1}>0\) in the interior of \(M\),
\[
 (\Psi_R(X))_{n+1}=\frac{x_{n+1}}{|\Pi_RX|},\qquad
 \Psi_R(\partial\Sigma_{r_0})\subset\partial H_R,\quad
 \Psi_R(\Sigma_{r_0}^\circ)\subset H_R^\circ.
\]
It is a covering map from a connected compact level set onto a simply
connected hemisphere, hence a diffeomorphism.
Writing \(\Pi_RX(r_0,\omega)=\rho(r_0,\omega)\omega\), we have
\[
 \partial_{r_0}\rho
 =\frac{\rho r_0}{r_0^2-\frac{(X\cdot\nu)(X\cdot\bar\nu_R)}{\bar\nu_R\cdot\nu}}
 =1+o(1)>0.
\]
Thus the portion of \(E\) in the annulus is the single-valued graph
\(X=x+u(x)\bar\nu_R\) over its image under \(\Pi_R\).
For an interior end \(E\), choose \(p_R\in E\), \(|p_R|=3R\), set
\(\nu_R=\nu(p_R)\), and use
the orthogonal projection \(\Pi_RX=X-(X\cdot\nu_R)\nu_R\) onto
\(\nu_R^\perp\). The same argument applies to the map
\(X\mapsto\frac{\Pi_RX}{|\Pi_RX|}\) onto \(\nu_R^\perp\cap S^n\cong S^{n-1}\).

We next use a standard Allard-type harmonic approximation argument
\cite{Allard72}; see also \cite[Chapter 5]{Simon83}.
For \(0<\alpha<n-1\), we prove
\begin{equation}\label{contraction}
 \sup_{E\cap\{r\ge LR\}}r|h|
 \le L^{-\alpha}\sup_{E\cap\{r\ge R\}}r|h|
 \qquad(R\ge R_0)
\end{equation}
with fixed \(L>6\). The linear estimate is
\begin{equation}\label{linear-decay}
 |D^2w|\le \frac{C}{r}\quad\Longrightarrow\quad
 |D^2w|\le C'r^{-n}\quad(r\ge6),
\end{equation}
where \(r=|x|\) and \(w\) is harmonic on \(\{|x|>2\}\), with
\(w(x_*)=0\), \(Dw(x_*)=0\), and \(|x_*|=3\).
The Hessian bound implies
\[
 \sup_{|x|=3}(|w|+|Dw|)\le C_n,\qquad w=O(r\log r).
\]
The exterior spherical harmonic expansion therefore gives
\(w=a_0+\ell\cdot x+O_2(r^{2-n})\). The bound at \(|x|=3\) controls the
\(O_2(r^{2-n})\) term. Differentiating twice gives \eqref{linear-decay}.
Here \(C_n\) and \(C'\) depend only on \(n\) and \(C\).
For harmonic \(w\) on \(\{x_1\ge0,|x|>2\}\) satisfying
\(\partial_1w=0\) on the boundary, even reflection across \(x_1=0\)
gives the same estimate.

It suffices first to show
\begin{equation}\label{eg:annular-improvement}
 \sup_{E\cap\{LS\le r\le2LS\}}r|h|
 \le L^{-\alpha}\sup_{E\cap\{r\ge S\}}r|h|
 \qquad(S\ge S_0).
\end{equation}
Suppose that \eqref{eg:annular-improvement} fails. Choose \(S_j\to\infty\)
with \(\epsilon_j=\sup_{E\cap\{r\ge S_j\}}r|h|>0\).
Apply the dilation \(X\mapsto S_j^{-1}X\) to the end \(E\). After a horizontal rotation, the unit normal at \(S_j^{-1}p_{S_j}\) is
\(\nu_0=\sin\theta\,\bar\nu_0+\cos\theta\,E_{n+1}\).
Choose orthonormal coordinates \(x=(x_1,\ldots,x_n)\)
on \(\nu_0^\perp\) such that
\(\langle x,E_{n+1}\rangle=\sin\theta\,x_1\).
Since \(\bar\nu_0\) is horizontal, the contact boundary
corresponds to \(x_1=0\). Write
\[
 X=x+u_j(x)\bar\nu_0,\qquad
 b_j=u_j(x_j),\qquad
 w_j=\frac{u_j-b_j}{\epsilon_j},
\]
where \(x_j=\Pi_0(\frac{p_{S_j}}{S_j})\) and
\(\Pi_0Y=Y-\frac{(Y\cdot\nu_0)\bar\nu_0}{\sin\theta}\) projects onto
\(\nu_0^\perp\) along \(\bar\nu_0\). Since \(\nu_0\) is the unit normal at \(S_j^{-1}p_{S_j}\),
we have \(Du_j(x_j)=0\). The definitions and
\eqref{normal-oscillation} give
\[
 \begin{gathered}
 b_j=o(1),\qquad w_j(x_j)=0,\qquad Dw_j(x_j)=0,\\
 |u_j-b_j|+|Du_j|\le C_B\epsilon_j\quad(2<|x|<B),\qquad
 |D^2w_j|\le \frac{C_\theta}{|x|}.
 \end{gathered}
\]
For every \(B>2\), the functions \(u_j\) and \(w_j\)
are defined on \(\{x_1\ge0,\ 2<|x|<B\}\)
for all sufficiently large \(j\).

The area integrand and the capillary minimal graph equations are
\begin{equation}\label{rescaled-graph-equations}
 \begin{gathered}
 \mathcal L(p)=
 \sqrt{1-2\cos\theta\,p_1+p_1^2
       +\sin^2\theta\sum_{i=2}^np_i^2},\\
 \operatorname{div}D\mathcal L(Du_j)=0\quad(x_1>0),\qquad
 D_1\mathcal L(Du_j)=-\cos\theta\quad(x_1=0).
 \end{gathered}
\end{equation}
Since \(D(u_j-b_j)=Du_j=\epsilon_jDw_j\),
\eqref{rescaled-graph-equations} gives
\[
 \operatorname{div}(a_jDw_j)=0,\qquad
 (a_jDw_j)_1=0\quad(x_1=0),\qquad
 a_j=\int_0^1D^2\mathcal L(\tau Du_j)\,d\tau.
\]
Since \(Du_j\to0\) and \(D^2\mathcal L(0)=\sin^2\theta I\),
we have \(a_j\to\sin^2\theta I\) locally uniformly.

The Hessian bound gives a uniform \(C^1\) bound for \(a_j\) on smaller annular regions, where the matrices are uniformly elliptic. Interior and boundary
Schauder estimates \cite[Theorem 7.3]{ADN59} give
\(C^2\) convergence, along a subsequence, to a function satisfying
\[
 \Delta w=0,\quad\partial_1w=0,\quad |D^2w|\le \frac{C_\theta}{|x|},
 \quad w(x_*)=0,\quad Dw(x_*)=0,\quad |x_*|=3.
\]
By \eqref{linear-decay}, \(\sup_{L\le|x|\le2L}|x||D^2w|\le C_0L^{1-n}\).
Choose \(L\) so large that the corresponding curvature bound satisfies
\[
 C_1L^{1-n}<\tfrac12L^{-\alpha}.
\]
Smooth convergence then contradicts the failure of
\eqref{eg:annular-improvement}.
For \(r(p)\ge LS\), applying \eqref{eg:annular-improvement} at
\(S'=\frac{r(p)}{L}\ge S\) gives
\(r(p)|h(p)|\le L^{-\alpha}\sup_{E\cap\{r\ge S\}}r|h|\).
Taking the supremum proves \eqref{contraction}.

For an interior end, write \(X=x+u_j(x)\nu_0\), where \(\nu_0\) is the
reference unit normal. The area integrand is \(\sqrt{1+|p|^2}\), and the
same rescaling gives \(\Delta w=0\) on
\(\{x\in\mathbb R^n:|x|>2\}\). Thus \eqref{linear-decay} gives the same contraction.
Iteration and the finiteness of the ends give, for every \(0<\alpha<n-1\),
\[
 |h|\le C_\alpha r^{-1-\alpha}\qquad(r\ge R_0).
\]

Choose \(1<\alpha<n-1\). Extrinsic Euclidean volume growth gives
\[
 \int_{M\setminus B_R}|h|^n
 \le C_\alpha\sum_{j=0}^\infty(2^jR)^{-n\alpha}
 \le C_\alpha R^{-n\alpha}.
\]
Thus \(M\) has finite total curvature.

We next extend the asymptotic analysis in
\cite[Section 2]{Schoen83} and \cite[Theorem A.4]{CG26}
to the capillary setting, proving that each end has a unique
tangent cone at infinity with multiplicity one.
Integrating the curvature estimate along paths in an end \(E\),
we obtain
\[
 \nu=\nu_\infty+O(r^{-\alpha}),\qquad
 \nu_\infty\cdot E_{n+1}=\cos\theta\quad\text{for a boundary end}.
\]
Set
$
 e_\infty=
 \begin{cases}
  \nu_\infty, & \text{for an interior end},\\
  \frac{\nu_\infty-\cos\theta\,E_{n+1}}{\sin\theta},
    & \text{for a boundary end}.
 \end{cases}
$
The preceding projection argument gives, outside a compact set,
\[
 X=x+u(x)e_\infty,\qquad
 |Du(x)|\le C|x|^{-\alpha},\qquad |x|\sim r.
\]
Here \(x\in\nu_\infty^\perp\), with \(x_{n+1}\ge0\) for a boundary end.
Since \(\alpha>1\),
\[
 \int_R^\infty\sup_{|x|=t}|Du(x)|\,dt
 \le CR^{1-\alpha},\qquad
 \operatorname{osc}_{|x|=R}u\le CR^{1-\alpha}\longrightarrow0.
\]
Thus \(u(x)\to b\) as \(|x|\to\infty\), with
\(|u-b|\le C|x|^{1-\alpha}\).
Interior and boundary Schauder estimates give
\begin{equation}\label{weak}
 D^j(u-b)=O(r^{1-\alpha-j})\qquad(j\ge0).
\end{equation}
Thus each end has a unique tangent cone at infinity,
with multiplicity one.
If the limiting hyperplane is not horizontal, then it contains a ray with
\(x_{n+1}\le-\gamma r\) for some \(\gamma>0\). Since \(u(x)\to b\), along this ray,
\[
 X\cdot E_{n+1}
 =x\cdot E_{n+1}+u(x)(e_\infty\cdot E_{n+1})
 \le-\gamma r+O(1)<0
\]
for sufficiently large \(r\), contradicting
\(M\subset\{x_{n+1}\ge0\}\). Thus interior ends are horizontal.
For a boundary end, write \(x'=(x_1,\ldots,x_{n-1})\) and
\[
 x_n=-\cot\theta\,x_{n+1}+b+v(x',x_{n+1}),\qquad
 V(x',\xi)=v(x',\sin\theta\,\xi),\qquad \rho=|(x',\xi)|.
\]
The minimal graph equation and contact angle condition give
\begin{equation}\label{end-PDE}
 |\Delta V|\le C|DV|\,|D^2V|,\qquad
 V_\xi=-\cos\theta\bigl(\sqrt{1+|D_{x'}V|^2}-1\bigr)\quad(\xi=0).
\end{equation}
Choose \(0<\eta<\frac{1}{2}\) and apply \eqref{weak} with
\(\alpha=n-1-\eta\). Then
\[
 \begin{gathered}
 V=O_k(\rho^{2-n+\eta}),\\
 \Delta V=O_k(\rho^{1-2n+2\eta}),\qquad
 V_\xi=O_k(\rho^{2-2n+2\eta}).
 \end{gathered}
\]

Choose a smooth radial function \(\chi\), equal to zero on an inner ball
and to one outside a larger ball. Define \(\widetilde V(x',\xi)=(\chi V)(x',|\xi|)\).
Then, in the sense of distributions,
\begin{equation}\label{reflected-end-source}
 \Delta\widetilde V=\lambda
       =F\,dx'\,d\xi+2G\,dx'\,\delta_{\{\xi=0\}},
 \qquad
 \int_{\mathbb R^n}(1+|\zeta|)\,d|\lambda|(\zeta)<\infty.
\end{equation}
Here \(F\) is the even extension of \(\Delta(\chi V)\), and
\(G=\partial_\xi(\chi V)|_{\xi=0}\).

The Newtonian potential representation is
\[
 \widetilde V(x)
 =-\frac{1}{(n-2)|S^{n-1}|}
       \int_{\mathbb R^n}|x-\zeta|^{2-n}\,d\lambda(\zeta).
\]
Both sides vanish at infinity. Their difference is harmonic on
\(\mathbb R^n\) and hence vanishes identically by the maximum principle.

For \(|\zeta|<\frac{\rho}{2}\),
\[
 \bigl||x-\zeta|^{2-n}-\rho^{2-n}\bigr|
      \le C\rho^{1-n}|\zeta|.
\]
Since \(\int|\zeta|\,d|\lambda|<\infty\), the error is \(O(\rho^{1-n})\) on
\(\{|\zeta|<\frac{\rho}{2}\}\) and on
\(\{|\zeta|\ge\frac{\rho}{2},\ |x-\zeta|\ge\frac{\rho}{4}\}\).
Near \(x\), we use
\[
 \begin{aligned}
 \int_{B_{\frac{\rho}{4}}(x)}|x-\zeta|^{2-n}\,d\zeta
 &\le C_n\rho^2,\\
 \int_{B_{\frac{\rho}{4}}(x)\cap\{\zeta_n=0\}}
 |x-\zeta|^{2-n}\,d\mathcal H^{n-1}(\zeta)
 &\le C_n\rho.
 \end{aligned}
\]
The decay of \(F\) and \(G\), together with rescaled Poisson and Neumann
estimates, gives
\[
 V=a\rho^{2-n}+O_k(\rho^{1-n}).
\]

Substituting this expansion into \eqref{end-PDE} gives
\[
 \Delta V=O_k(\rho^{1-2n}),\qquad V_\xi=O_k(\rho^{2-2n}).
\]
Consequently,
\[
\int_{B_R}|\zeta|^2\,d|\lambda|(\zeta)
 \le C\left(1+\int_1^R s^{2-n}\,ds\right)
 \le
 \begin{cases}
  C(1+\log R), & n=3,\\
  C, & n\ge4.
 \end{cases}
\]
The error after expanding the Newtonian kernel to first order on
\(|\zeta|<\frac{\rho}{2}\) is bounded by
\(C\rho^{-n}\int_{B_{\frac{\rho}{2}}}|\zeta|^2\,d|\lambda|\).
The integral over \(|\zeta|\ge\frac{\rho}{2}\) is \(O(\rho^{3-2n})\), using
the ball and hyperplane kernel bounds near \(\zeta=x\).
Since \(\lambda\) is even in \(\xi\), \(\int\xi\,d\lambda=0\), so the
first-order Taylor term is \((d\cdot x')\rho^{-n}\), with
\(d\in\mathbb R^{n-1}\). Returning to \((x',x_{n+1})\) and applying
rescaled Poisson and Neumann estimates proves the boundary expansion
\eqref{boundary-exp} with all derivatives.

For an interior end, write \(x'=(x_1,\ldots,x_n)\) and
\(x_{n+1}=b+v(x')\).
Here \(\lambda=F\,d\zeta\), and the same argument gives
\[
 v(x')=a|x'|^{2-n}+O_k(|x'|^{1-n}).
\]
The minimal equation then gives
\[
 \Delta v=\frac{D^2v(Dv,Dv)}{1+|Dv|^2}
               =O_k(|x'|^{2-3n}).
\]
Therefore
\[
 \int_{\mathbb R^n}|\zeta|^2|\Delta(\chi v)(\zeta)|\,d\zeta<\infty.
\]
Taylor expansion of \(|x'-\zeta|^{2-n}\) in the integral gives
\[
 v(x')=a|x'|^{2-n}+\frac{d\cdot x'}{|x'|^n}+E(x'),
 \qquad E=O(|x'|^{-n}).
\]
Interior estimates for \(\Delta E=\Delta v\), with \(x'=Ry\),
give \(|D^jE|=O(|x'|^{-n-j})\) for every \(j\ge0\).
This proves \eqref{interior-exp}. The derivative bounds for the two end graphs give
\(|\nabla^jh|=O(r^{-n-j})\) for every \(j\ge0\), completing the proof.
\end{proof}

\section{The higher-dimensional classification}\label{sec:classification-three}

Let \(M^n\), \(n\ge3\), satisfy the hypotheses of Theorem~\ref{main}.

Let \(\mathcal D^{1,2}(M)\) be the completion of \(C_c^\infty(M)\) in the
Dirichlet norm, and let \(\mathcal D^{1,2}_0(M)\) denote the
subspace of functions with zero trace on \(\partial M\).
Theorem~\ref{main} and the boundary area bound in
Section~\ref{sec:structure} give
\[
 |h|^2\in L^{\frac{n}{2}}(M),\qquad q\in L^{n-1}(\partial M).
\]
H\"older's inequality and \eqref{Sob} bound \(Q\) on
\(\mathcal D^{1,2}(M)\). Rellich's theorem and compact trace embeddings give
compact operators for bounded compactly supported coefficients.
Approximation in \(L^{\frac{n}{2}}(M)\) and \(L^{n-1}(\partial M)\) yields
\[
 Q(f,g)=\langle(I-K)f,g\rangle_{\mathcal D^{1,2}(M)},
 \qquad K\text{ compact and self-adjoint}.
\]
By density, the index on the Dirichlet space equals \(\Ind_\theta(M)\).

If \(\Ind_\theta(M)=1\), fix \(\psi\) with \(Q(\psi,\psi)<0\).
Then \(Q(f,\psi)=0\) implies \(Q(f,f)\ge0\), and polarization gives
\begin{equation}\label{index-one-orthogonality}
 Q(f,\psi)=Q(f,f)=0
 \quad\Longrightarrow\quad
 Q(f,\eta)=0\quad\text{for every }\eta\in\mathcal D^{1,2}(M).
\end{equation}
Such \(f\) satisfies \(\Delta f+|h|^2f=0\) and
\(\partial_\mu f=qf\).

\begin{prop}\label{compact-index}
Suppose that \(\partial M\) is compact and \(\Ind_\theta(M)=1\). Then
\[
 N=B=b_{n-1}(M)=1,
\]
where \(N\) and \(B\) are the numbers of ends and boundary components.
\end{prop}
\begin{proof}
The ends are horizontal graphs by Theorem~\ref{main}:
\[
 x_{n+1}=b_\alpha+a_\alpha r^{2-n}+O_k(r^{1-n})
 \quad\text{on }E_\alpha.
\]
The harmonic function \(x_{n+1}\) is positive and bounded; the maximum
principle gives \(b_\alpha>0\) for some \(\alpha\).

For \(G\in C_c^\infty(M)\), the Lax--Milgram theorem and \eqref{Sob} give
a unique solution
\[
 \Delta w=G,\qquad w\in\mathcal D^{1,2}_0(M),\qquad
 w=O_k(r^{2-n}).
\]
To prove the decay estimate, choose regular radii \(R_j\uparrow\infty\)
with \(\partial M\cup\operatorname{supp}G\subset B_{R_1}\), and set
\[
 \Omega_j=M\cap B_{R_j},\qquad \Gamma_j=M\cap\partial B_{R_j}.
\]
Let \(w_j\) solve the Dirichlet problem
\[
 \begin{cases}
 \Delta w_j=G & \text{in }\Omega_j,\\
 w_j=0 & \text{on }\partial M\cup\Gamma_j.
 \end{cases}
\]
For \(\phi=(1+|X|^2)^{\frac{2-n}{2}}\), we have
\[
 \Delta\phi
 =-n(n-2)(1+|X^\perp|^2)(1+|X|^2)^{-\frac{n+2}{2}}.
\]
Since \(G\) is compactly supported, choose \(C>0\), independent of \(j\),
such that \(-C\Delta\phi\ge|G|\). The maximum principle gives
\[
 |w_j|\le C\phi=C(1+|X|^2)^{\frac{2-n}{2}}\quad\text{on }\Omega_j.
\]
Extend \(w_j\) by zero to \(M\). Testing the equation with \(w_j\) and
using \eqref{Sob} gives a uniform Dirichlet bound, so a subsequence satisfies
\[
 w_j\rightharpoonup w\quad\text{in }\mathcal D^{1,2}_0(M).
\]
Uniqueness identifies the limit with the solution above. The pointwise
bound passes to \(w\), and rescaling on the ends gives the derivative estimates.

Choose a compactly supported closed relative representative of
\([\alpha]\in H_c^1(M,\partial M;\mathbb R)\). Minimizing
\(\int_M|\alpha+df|^2\) over \(f\in\mathcal D^{1,2}_0(M)\) gives
\[
 \omega=\alpha+df,\qquad d\omega=\delta\omega=0,\qquad
 \omega=O_k(r^{1-n}).
\]
If \(\omega=0\), then \(f\) is constant on each end;
\(f\in L^{\frac{2n}{n-2}}(M)\) forces these constants to vanish. Thus \(f\)
is a compactly supported relative primitive. Let \(\mathcal H_0\) denote
the space of \(L^2\) harmonic one-forms obtained by this construction,
with relative boundary condition \(\omega|_{T\partial M}=0\).
Duality gives \(\dim\mathcal H_0=b_{n-1}(M)\).
Choose \(\chi=0\) near \(\partial M\) and \(\chi=b_\alpha\) on each
\(E_\alpha\). Then \(d\chi\) is a compactly supported closed relative
one-form and \(x_{n+1}-\chi\in\mathcal D^{1,2}_0(M)\).
Since \(x_{n+1}\) is harmonic, \(d\chi+d(x_{n+1}-\chi)=dx_{n+1}\)
is its minimizing harmonic representative, so \(dx_{n+1}\in\mathcal H_0\).

For \(v\in E_{n+1}^\perp\), choose \(F_v\), depending linearly on \(v\),
vanishing near the boundary and equal to
\(b_\alpha\langle v,X\rangle\) on each end. A Dirichlet correction with
Laplacian \(-\Delta F_v\) gives relative harmonic one-forms \(\omega_v\).
Define
\[
 \mathcal H=\mathcal H_0\oplus\{\omega_v:v\in E_{n+1}^\perp\},\qquad
 \dim\mathcal H=b_{n-1}(M)+n.
\]
The sum is direct because some \(b_\alpha>0\).
For \(\omega\in\mathcal H\), let \(v\in E_{n+1}^\perp\) satisfy
\(\omega-\omega_v\in\mathcal H_0\). Then, on \(E_\alpha\),
\[
 \omega^\sharp=b_\alpha v^\top+O_k(r^{1-n}).
\]

The dilation Jacobi field \(Z=X\cdot\nu\) satisfies
\[
 \begin{gathered}
 \Delta Z+|h|^2Z=0,\qquad \partial_\mu Z=qZ,\\
 Z=\sigma_\alpha\bigl(b_\alpha+(n-1)a_\alpha r^{2-n}\bigr)
       +O_k(r^{1-n}),
 \end{gathered}
\]
where \(\sigma_\alpha\) is the orientation sign of the end. Green's formula
gives
\[
 Q(Z,Z)=(n-1)\int_M|dx_{n+1}|^2>0,\qquad
 Q(Z,f)=0\quad(f\in\mathcal D^{1,2}(M)).
\]
The integrals converge even if \(Z\notin\mathcal D^{1,2}(M)\).
For the first identity, cancel the Robin boundary terms and use
\[
 Z\partial_\eta Z=(n-1)x_{n+1}\partial_\eta x_{n+1}+O(r^{-n}).
\]
For the second, use density, \(\nabla Z,|h|Z\in L^2(M)\), and the trace
estimate.

For a constant \(\Omega\in\Lambda^{n-3}\mathbb R^{n+1}\), consider an
admissible pairing in the sense of \cite[Section 6]{CG26}:
\[
 P_\Omega\beta=\beta+\star(\Omega\wedge\beta),\qquad
 P_\Omega=\sum_a\Theta_a\otimes\Theta_a\ge0.
\]
Here \(\star\) is the Euclidean Hodge star and the \(\Theta_a\) are constant
two-forms. For decomposable \(\beta\),
\[
 \langle P_\Omega\beta,\beta\rangle=|\beta|^2.
\]
Let \(\Omega_M\) be the pullback of \(\Omega\) to \(M\).
For \(u_a=\langle\nu\wedge\omega^\sharp,\Theta_a\rangle\),
\cite[Lemma 7.1]{CG26} gives
\begin{equation}\label{pairing}
 \sum_a(|\nabla u_a|^2-|h|^2u_a^2)\,dA
 =\tfrac12\Delta|\omega|^2\,dA
  +2(-1)^{n+1}d(\Omega_M\wedge\omega\wedge h\omega).
\end{equation}
Here \((h\omega)(U)=\omega(h^\sharp U)\); the sign in
\eqref{pairing} follows from our convention for \(h\).

On \(\partial M\), write \(\omega^\sharp=f\mu\). The identities
\(d\omega=\delta\omega=0\) and \eqref{boundary} give
\begin{equation}\label{relative-boundary}
 \begin{gathered}
 \nabla_\mu\omega^\sharp=\nabla^{\partial M}f+qf\mu,\qquad
 \tfrac12\partial_\mu|\omega|^2=q|\omega|^2,\\
 (\bar\nabla_\mu-q)(\nu\wedge\omega^\sharp)
 =\nu\wedge\nabla^{\partial M}f.
 \end{gathered}
\end{equation}
Integration by parts in \eqref{pairing}, together with
\eqref{relative-boundary}, cancels the Robin boundary term.
The form \(\Omega_M\wedge\omega\wedge h\omega\) has zero pullback to
\(\partial M\), and both boundary integrals over
\(M\cap\partial B_R\) are \(O(R^{-1})\).
Thus we obtain
\[
 \sum_a Q(u_a,u_a)=0,\qquad
 \widetilde u_a=u_a-\langle E_{n+1}\wedge v,\Theta_a\rangle Z
       =O_k(r^{2-n})\in\mathcal D^{1,2}(M),
\]
and
\begin{equation}\label{subtraction}
 \sum_a Q(\widetilde u_a,\widetilde u_a)
 =-(n-1)|v|^2\int_M|dx_{n+1}|^2.
\end{equation}

Choose \(\psi\in C_c^\infty(M)\) with \(Q(\psi,\psi)<0\).
If \(Q(u_a,\psi)=0\) for all \(a\), then
\(Q(\widetilde u_a,\psi)=0\), since \(Q(Z,\psi)=0\).
Index one and \eqref{subtraction} give
\[
 0\le\sum_a Q(\widetilde u_a,\widetilde u_a)
 =-(n-1)|v|^2\int_M|dx_{n+1}|^2.
\]
Hence \(v=0\) and each \(Q(u_a,u_a)=0\). By
\eqref{index-one-orthogonality}, each \(u_a\) satisfies the Jacobi equation
and Robin boundary condition.
Equation~\eqref{relative-boundary} gives
\[
 P_\Omega(\nu\wedge\nabla^{\partial M}f)=0.
\]
Since this two-form is decomposable, \(\nabla^{\partial M}f=0\).

On a connected boundary patch, \(f=f_0\). Set
\(\tau=\omega+(\frac{f_0}{\sin\theta})dx_{n+1}\). Since \(\tau=0\) there,
its tangential derivatives vanish; \(d\tau=\delta\tau=0\) give
\(\nabla_\mu\tau=0\). Boundary and interior unique
continuation~\cite[Lemma 2]{CDGM05} yield \(\tau=0\) on \(M\), hence
\(\omega\in\operatorname{span}\{dx_{n+1}\}\).

Define
\[
 \mathcal T:\mathcal H\longrightarrow\Lambda^2\mathbb R^{n+1},\qquad
 \langle\mathcal T(\omega),\Theta\rangle
 =Q(\langle\nu\wedge\omega^\sharp,\Theta\rangle,\psi).
\]
Suppose that \(b_{n-1}(M)\ge2\). Choose \(\mathcal V\subset\mathcal H\)
with
\[
 \dim\mathcal V=n+1,\qquad
 \mathcal V\cap\operatorname{span}\{dx_{n+1}\}=\{0\}.
\]
If \(\mathcal T|_{\mathcal V}\) has nonzero kernel, take \(0\ne\omega\)
in that kernel and \(P_\Omega=\operatorname{Id}\).
Otherwise \(\dim\mathcal T(\mathcal V)=n+1\), and
\cite[Propositions 6.7--6.8]{CG26} give \(0\ne\omega\in\mathcal V\)
and an admissible \(P_\Omega\) with \(P_\Omega\mathcal T(\omega)=0\).
In both cases,
\[
 0=\langle P_\Omega\mathcal T(\omega),\mathcal T(\omega)\rangle
   =\sum_a|Q(u_a,\psi)|^2.
\]
The preceding argument gives
\(\omega\in\operatorname{span}\{dx_{n+1}\}\), a contradiction.
Thus \(b_{n-1}(M)\le1\).

Truncating the ends gives a compact manifold with \(B+N\) boundary
components. The exact sequence of the pair and duality give
\(b_{n-1}(M)\ge B+N-1\ge1\). Hence \(N=B=b_{n-1}(M)=1\).
\end{proof}

At an interior contact point, the strong maximum principle implies that
the two local graphs agree.
At a boundary contact point away from the reflecting hyperplane, choose
coordinates centered at the contact point so that the common outward
boundary normal within \(\{x_{n+1}=0\}\) is \(E_n\).
Set \(z=x_{n+1}\) and write
\[
 X_i=(x',g_i(x',z),z),\qquad i=1,2,\qquad
 W=g_2-g_1\ge0,\qquad W(0)=0.
\]
Here \(g_1\) and \(g_2\) describe the original and reflected pieces,
respectively.
The capillary boundary condition is
\begin{equation}\label{reflection-angle}
 (g_i)_z=-\cot\theta\sqrt{1+|D_{x'}g_i|^2}
 \qquad(z=0).
\end{equation}
Subtracting the minimal graph equations and linearizing the boundary
condition gives
\[
 \begin{cases}
 \operatorname{div}(A\,DW)=0,&z>0,\\
 W_z+\beta\cdot D_{x'}W=0,&z=0,
 \end{cases}
\]
where \(A\) is uniformly elliptic and, with \(g_\tau=g_1+\tau W\),
\[
 \beta=\cot\theta\int_0^1
 \frac{D_{x'}g_\tau}{\sqrt{1+|D_{x'}g_\tau|^2}}\,d\tau.
\]
Since \(W\ge0\) and \(W(0)=0\), the boundary condition yields
\[
 D_{x'}W(0)=0,\qquad
 W_z(0)=-\beta(0)\cdot D_{x'}W(0)=0.
\]
If \(W\not\equiv0\), the strong maximum principle and Hopf's lemma give
\[
 W_z(0)>0,
\]
a contradiction. Hence the two graphs agree near the contact point.

At a corner contact point, let \(g\) denote the original local graph function and
choose coordinates so that the reflecting plane is \(x_1=0\).
Its normal is tangent to \(M\) and \(\partial M\), and
\(D_{x'}g=0\) at the contact point. The graph difference
\[
 W=g(-x_1,x_2,\ldots,x_{n-1},x_{n+1})
     -g(x_1,x_2,\ldots,x_{n-1},x_{n+1})\ge0
\]
is odd in \(x_1\), with \(W=DW=0\) at the contact point. Only second
derivatives \(W_{1j}\) can be nonzero; those tangent to the support vanish
at the minimum of \(W_1\) on the intersection of the support and the
reflecting plane. Equation~\eqref{reflection-angle} gives
\(g_{x_1x_{n+1}}=0\) at the contact point. The coefficient matrix is
$
 A^{ij}=\int_0^1
 \left.\frac{\partial}{\partial p_j}
       \frac{p_i}{\sqrt{1+|p|^2}}\right|_{p=Dg+tDW}\,dt.
$
Thus
\[
 D^2W=0,\qquad \partial_i(A^{ij}\partial_jW)=0,\qquad
 A^{1j}=0\quad(x_1=0, j\ne1).
\]
The coefficients have bounded first derivatives. Serrin's corner lemma
\cite[Section 4, Lemma 2]{Serrin71} implies, for an inward direction \(\eta\),
\[
 W\not\equiv0\quad\Longrightarrow\quad
 (D_\eta W,D^2_{\eta\eta}W)(0)\ne(0,0),
\]
contrary to \(DW(0)=D^2W(0)=0\). Hence \(W\equiv0\).

\begin{proof}[Proof of Theorem~\ref{index-one}, compact boundary]
Proposition~\ref{compact-index} gives one end and one boundary component.
The maximum principle, the divergence theorem for \(x_{n+1}\), and
\(\partial_\mu x_{n+1}=-\sin\theta\) give
\[
 \begin{gathered}
 0<x_{n+1}<b,\qquad x_{n+1}=b-ar^{2-n}+O_k(r^{1-n}),\\
 (n-2)|S^{n-1}|a=\sin\theta\,|\partial M|>0.
 \end{gathered}
\]
After a unique horizontal translation, \(d=0\) in \eqref{interior-exp}.
Write \(u=b-ar^{2-n}+H\), \(H=O_k(r^{-n})\).

Let \(D\subset\{x_{n+1}=0\}\) be the bounded region enclosed by \(\partial M\).
On the graphical end, the change of coordinates
\[
 (x',t)\longmapsto(x',t-u(x')+b)
\]
maps the end to the plane \(t=b\). Extend this change smoothly to equal
the identity near \(\bar D\). Adding the point at infinity then gives a
closed embedded hypersurface in \(\mathbb S^{n+1}\); the separation theorem
therefore shows that \(M\cup\bar D\) separates \(\mathbb R^{n+1}\).
For sufficiently large
\(|x'|\), the line \(\{(x',t):t\in\mathbb R\}\) crosses the end once,
so the component \(\Omega\) containing
\(\{x_{n+1}>b\}\) satisfies
\[
 \Omega\subset\{x_{n+1}>0\},\qquad
 \Omega=\{x_{n+1}>u(x')\},\qquad x'=(x_1,\ldots,x_n),
 \quad\text{outside a cylinder}.
\]

Write \(\widehat x=(x_2,\ldots,x_n)\), and let
\(\mathcal R_t(x_1,\widehat x,x_{n+1})=(2t-x_1,\widehat x,x_{n+1})\)
denote reflection across \(x_1=t\). Reflect
\(\Omega\cap\{x_1>\lambda\}\) by \(\mathcal R_\lambda\).
For fixed \(\lambda_0>0\), the
centered expansion gives, uniformly at infinity for \(\lambda\ge\lambda_0\),
\[
 u(\lambda+\ell,\widehat x)>u(\lambda-\ell,\widehat x),\qquad
 u_1(\lambda,\widehat x)>0.
\]
Indeed, with \(r_\pm=|(\lambda\pm\ell,\widehat x)|\),
\[
 \begin{gathered}
 r_+^2-r_-^2=4\lambda\ell,\\
 a(r_-^{2-n}-r_+^{2-n})\ge
 \begin{cases}
 a(1-2^{2-n})r_-^{2-n},&r_+\ge2r_-,\\
 \gamma_na\lambda\ell r_-^{-n},&r_+<2r_-.
 \end{cases}
 \end{gathered}
\]
For \(r_-\) large, \(H\) satisfies
\[
 \frac{|H(\lambda+\ell,\widehat x)-H(\lambda-\ell,\widehat x)|}
      {a(r_-^{2-n}-r_+^{2-n})}
 \le C\left(r_-^{-2}+\frac1{\lambda_0r_-}\right)\longrightarrow0.
\]
If the reflected point lies in \(\{|x'|\le R_0\}\), then use
\[
 \sup_{M\cap\{|x'|\le R_0\}}x_{n+1}\le b-\eta,\qquad
 u(x')\longrightarrow b.
\]
Thus reflection starts for large \(\lambda\).

Define
\[
 \lambda_*=\inf\bigl\{s\ge0:
 \mathcal R_t(\Omega\cap\{x_1>t\})\subset\Omega
 \text{ for every }t>s\bigr\}.
\]
The preceding estimates ensure that this set is nonempty.
Suppose that \(\lambda_*>0\). The end estimates exclude contact
at infinity, so a first contact occurs at some \(q\in M\).
For the local graph difference \(W=g_{\mathrm{ref}}-g\ge0\),
with \(W(q)=0\), the interior, boundary, and corner comparisons above
give \(W\equiv0\).
At an interior point on the reflecting plane, orthogonality gives
\(\partial_{\mathrm{in}}W(q)=0\), and Hopf's lemma gives
the same conclusion.
Thus the two pieces agree near \(q\). Analytic continuation and properness
give reflection symmetry of \(M\) across \(x_1=\lambda_*\).
Uniqueness of the translation making \(d=0\) in
\eqref{interior-exp} then forces \(\lambda_*=0\).
Applying the argument from both sides in every horizontal direction, we
conclude that \(M\) is invariant under \(O(n)\).

The rotational classification~\cite[Section 2]{TZ09} places the nonflat
\(M\) in an \(n\)-dimensional catenoid; completeness and properness make
\(M\) a truncation of it.
\end{proof}

\begin{proof}[Proof of Theorem~\ref{index-one} for noncompact boundary]
For \(\theta=\frac{\pi}{2}\),
\[
 u:=\nu_{n+1}\in\mathcal D^{1,2}_0(M),\qquad
 u\not\equiv0,\qquad
 (\Delta+|h|^2)u=0.
\]
If \(Q(\phi,\phi)<0\) for some
\(\phi\in\mathcal D^{1,2}_0(M)\), then
\eqref{index-one-orthogonality} and unique continuation
for \(u\) extended by zero across \(\partial M\) \cite{Aro57} give
\[
 Q(u,u)=Q(u,\phi)=0
 \ \Longrightarrow\
 u=\partial_\mu u=0\quad\text{on }\partial M
 \ \Longrightarrow\ u\equiv0,
\]
a contradiction. Hence \(\operatorname{Ind}_D(M)=0\).
The reflected hypersurface has index
\(\Ind_{\frac{\pi}{2}}(M)+\operatorname{Ind}_D(M)=1\)
and is a catenoid by \cite[Theorem 1.1]{CG26}.
Since \(\partial M\) is noncompact, the plane \(\{x_{n+1}=0\}\)
contains its rotation line. Thus the normalized hypersurface is \(C^+\).

For fixed \(n\), set
\[
 I_n:=\{\kappa=-\cot\theta\ge0:
       (n,\theta)\text{ satisfies }\eqref{stable-cone-range}\}.
\]
For \(\kappa>0\), Theorem~\ref{main} and
Proposition~\ref{two-ends} give the normalization and symmetries
\[
 b=1,\qquad
 (x',x_n)\mapsto(Ox',\pm x_n),\quad O\in O(n-1),
\]
with end planes \(x_n=\pm(1+\kappa x_{n+1})\).
Let \(\mathscr M_1\) be the space of these normalized index-one
hypersurfaces with \(\kappa\in I_n\), including \(C^+\) at \(\kappa=0\).
Use weighted graph convergence after subtracting the affine end functions. Consider
\[
 \pi:\mathscr M_1\longrightarrow I_n,\qquad \pi(M)=\kappa.
\]
By Propositions~\ref{angle-chart} and~\ref{properness}, \(\pi\) is a proper
local homeomorphism. Its image is open and closed in \(I_n\) and contains
zero. Thus \(\pi\) is a finite covering of \(I_n\) with constant fiber
size. Since \(\pi^{-1}(0)=\{C^+\}\),
\[
 \#\pi^{-1}(\kappa)=1\qquad(\kappa\in I_n).
\]
The constructed models have index one by Proposition~\ref{hairpin-existence}
and therefore belong to this space.
Thus every hypersurface in the classification is one of these models,
up to dilations and rigid motions.
\end{proof}

\appendix
\section{Higher-dimensional capillary hairpins}
\label{sec:hairpin-construction}\label{sec:index-calculations}

In this appendix, let \(M^n\subset\mathbb R^{n+1}_+\), \(n\ge3\), be a
connected, complete, two-sided, properly embedded capillary minimal
hypersurface with noncompact boundary and finitely many ends. Assume that
each end is a single graph satisfying
\eqref{interior-exp} or \eqref{boundary-exp}, with the corresponding
curvature estimates. We use the compact-operator representation of \(Q\)
and the reflection arguments from Section~\ref{sec:classification-three}.

\begin{prop}\label{hairpin-existence}
For every \(n\ge3\) and \(\theta\in(0,\pi)\), there exists a connected,
complete, two-sided, properly embedded capillary minimal hypersurface
\(M^n_\theta\subset\mathbb R^{n+1}_+\) with noncompact boundary, exactly
two regular boundary ends, finite total curvature, extrinsic Euclidean
volume growth, and \(\Ind_\theta(M_\theta)=1\).
After fixing the center and scale, we obtain a smooth family
\(M_\kappa\), \(\kappa=|\cot\theta|\in[0,\infty)\),
with \(M_0=C^+\).
\end{prop}

The construction starts from \(C^+\), half of the \(n\)-dimensional
catenoid cut by a hyperplane containing its rotation line.
For \(n\ge3\) and \(\theta\ne\frac{\pi}{2}\), set \(\kappa=|\cot\theta|\) and normalize the
two ends to be asymptotic to \(x_n=\pm(1+\kappa x_{n+1})\).
Writing \(x'=(x_1,\ldots,x_{n-1})\), the normalized hypersurface
\(M_\theta\) has the parametrization
\begin{equation}\label{hairpin-graph}
 X_\theta(x',x_n)=\bigl(x',x_n,v_\theta(x',x_n)\bigr),\qquad
 (x',x_n)\in\overline{\Omega_\theta},
\end{equation}
where
\begin{equation}\label{hairpin-domain}
 \Omega_\theta
 =\{(x',x_n):|x_n|<1,\ |x'|<R_\theta(x_n)\}
  \cup\{(x',x_n):|x_n|\ge1\}.
\end{equation}
The function \(R_\theta:(-1,1)\to(0,\infty)\) satisfies
\[
 R_\theta(-t)=R_\theta(t),\qquad
 R_\theta'(t)>0\quad(0<t<1),\qquad
 \lim_{t\to\pm1}R_\theta(t)=\infty.
\]
The graph function satisfies
\begin{equation}\label{hairpin-graph-equation}
 \begin{gathered}
 \begin{cases}
 \displaystyle\operatorname{div}
       \frac{Dv_\theta}{\sqrt{1+|Dv_\theta|^2}}=0,
 \quad v_\theta>0&\text{in }\Omega_\theta,\\[2mm]
 v_\theta=0,\quad |Dv_\theta|=|\tan\theta|
 &\text{on }\partial\Omega_\theta,
 \end{cases}\\[2mm]
 v_\theta(Ox',x_n)=v_\theta(x',x_n)=v_\theta(x',-x_n)
 \quad(O\in O(n-1)).
 \end{gathered}
\end{equation}

\begin{prop}\label{two-ends}
Assume also that \(M\) is nonflat and
\(\theta\ne\frac{\pi}{2}\).
Choose the normal so that \(\cos\theta<0\), and set
\(\kappa=-\cot\theta\).
After a rigid motion preserving \(\{x_{n+1}=0\}\),
the following hold.

\smallskip\noindent\textup{(i)}
\(M\) has exactly two ends, both boundary ends, and is symmetric
under \(O(n-1)\) acting on \(x'\) and \(x_n\mapsto-x_n\).
The ends \(x_n=u_\pm(x',x_{n+1})\) satisfy
\begin{equation}\label{main-expansion}
 \begin{gathered}
 u_+=b+\kappa x_{n+1}-m\rho^{2-n}
 +
 \begin{cases}
  O_k(\rho^{-3}\log\rho), & n=3,\\
  O_k(\rho^{-n}), & n\ge4,
 \end{cases}\\
 u_-=-u_+,\qquad
 \rho^2=|x'|^2+\frac{x_{n+1}^2}{\sin^2\theta},
 \qquad b,m>0.
 \end{gathered}
\end{equation}
Moreover,
\[
 |x_n|<b+\kappa x_{n+1},\qquad
 x_n\nu_n>0\ (x_n\ne0),\qquad
 \nu_r<0\ (r=|x'|>0).
\]

\smallskip\noindent\textup{(ii)}
\(M\) is the graph \(x_{n+1}=v(x',x_n)\) satisfying
\eqref{hairpin-graph-equation} with \(v_\theta,\Omega_\theta\)
replaced by \(v,\{v>0\}\), and
\(|Dv|<|\tan\theta|\) in \(\{v>0\}\).
For \(S=M\cap\{x_n=0\}\), we have
\(S\cong\overline D^{\,n-1}\) and
\(M\cong\mathbb R\times\overline D^{\,n-1}\).
\end{prop}
\begin{proof}
Let \(\Omega\) be the region bounded by \(M\), with outward unit
normal \(\nu\).
A boundary end \(x_n=u(x',x_{n+1})=\kappa x_{n+1}+b+o(1)\)
and an interior end \(x_{n+1}=f(x',x_n)=b_0+o(1)\) would intersect:
for large \(|x'|\),
\[
 x_{n+1}-f(x',u(x',x_{n+1}))
\]
has opposite signs at \(x_{n+1}=0\) and \(x_{n+1}=b_0+1\).
Two boundary ends with nonparallel horizontal normals would also
intersect, since their limiting half-hyperplanes meet transversely
in \(x_{n+1}>0\).
Thus all ends are boundary graphs, which we order as
\[
 u_1<\cdots<u_N,\qquad
 u_i=\sigma_i\kappa x_{n+1}+b_i+o(1),\qquad
 \sigma_i\in\{-1,1\}.
\]
Along a line directed by \(E_n\), entering \(\Omega\) gives
\(\nu_n<0\), while leaving gives \(\nu_n>0\).
Hence \(\sigma_{i+1}=-\sigma_i\).
The ordering for large \(x_{n+1}\) gives
\(\sigma_i\le\sigma_{i+1}\), so \(N\le2\).
Since \eqref{mono} excludes \(N=1\) for nonflat \(M\),
\[
 N=2,\qquad
 u_\pm=\pm\kappa x_{n+1}+b_\pm+o(1),\qquad b_-\le b_+.
\]

The region \(\Omega\) lies between the two ends. Set
\[
 L_+=\sin\theta(x_n-b_+)+\cos\theta\,x_{n+1},\qquad
 L_-=-\sin\theta(x_n-b_-)+\cos\theta\,x_{n+1}.
\]
Since \(\Delta_M L_\pm=0\) and \(\limsup_\infty L_\pm\le0\),
a positive maximum of \(L_+\) occurs at \(p\in\partial M\), where
\(x_n(p)=\max_{\partial M}x_n\).
The half-space \(\{x_{n+1}=0,\ x_n>x_n(p)\}\) is connected
and disjoint from \(\partial M\); the end expansions show
that it lies outside \(\overline\Omega\). Hence
\[
 \bar\nu(p)=E_n,\qquad
 \partial_\mu L_+(p)
 =\sin\theta\cos\theta-\cos\theta\sin\theta=0,
\]
contradicting Hopf's lemma. Applying the same argument to \(L_-\) and
to a zero maximum gives
\begin{equation}\label{wedge}
 b_--\kappa x_{n+1}<x_n<b_++\kappa x_{n+1},\qquad b_-<b_+.
\end{equation}

Reflect \(\Omega\) across \(x_n=\lambda\), starting from above, and let
\(\lambda_0=\frac{b_++b_-}{2}\). Estimate~\eqref{wedge} allows reflection to begin on the ends. For each fixed \(\epsilon>0\), the following estimates hold uniformly at infinity:
\[
 u_-<\lambda_0,\qquad u_-+u_+<2\lambda_0+\epsilon,\qquad
 \nu_n>\frac{\sin\theta}{2}\quad\text{on the upper end}.
\]
Thus, for \(\lambda>\lambda_0\), first contact occurs in a compact set.
The comparison principle would give symmetry, which requires
\(2\lambda-b_+=b_-\), a contradiction. Reflection from both sides
therefore gives symmetry at \(\lambda_0\). For every \(\lambda>\lambda_0\), reflection in \(x_n=\lambda\)
maps \(\Omega\cap\{x_n>\lambda\}\) into \(\Omega\).
Together with symmetry about \(x_n=\lambda_0\), this shows
that each nonempty intersection of \(\Omega\) with a line
parallel to \(E_n\) is an interval.
They also give the strict signs of \(\nu_n\). Translate in \(x_n\) so that
\[
 \lambda_0=0,\qquad b_+=-b_-=b>0.
\]

In the coordinate \(\xi=\frac{x_{n+1}}{\sin\theta}\), the positive function
\(V=b+\kappa x_{n+1}-u_+\) satisfies
\[
 LV=0,\quad L=\Delta+O_k(\rho^{1-n})D^2,\qquad
 B_0V=0,\quad B_0=\partial_\xi+O_k(\rho^{1-n})D_{x'}.
\]
For \(H=\rho^{2-n}+\rho^{1-n}+\xi\rho^{-n}\),
\[
 LH=(n-1)\rho^{-n-1}+O(\rho^{-n-2})>0,\qquad
 B_0H=\rho^{-n}+O(\rho^{-n-1})>0.
\]
Choose \(\epsilon>0\) with \(V\ge\epsilon H\) on an inner hemisphere.
Since \(L(V-\epsilon H)<0\), \(B_0(V-\epsilon H)<0\), and
\(V-\epsilon H\to0\) at infinity, the maximum principle and Hopf's
lemma give \(V\ge\epsilon H\), hence \(m>0\).
A unique translation in the \(x'\)-variables eliminates the terms
\(\frac{d\cdot x'}{\rho^n}\) on both ends, giving \eqref{main-expansion}.

Set \(\mathcal E=\{x_{n+1}>0\}\setminus\overline\Omega\).
Repeat the moving-plane argument from the compact-boundary proof of
Theorem~\ref{index-one}, reflecting \(\mathcal E\) across \(x_1=\lambda\).
For \(\lambda\ge\lambda_0>0\), \eqref{main-expansion} gives the same
end comparisons, with error ratio \(O(\frac{\log\rho}{\lambda_0\rho})\)
for comparable radii; \eqref{wedge} gives separation when the reflected
base point lies in a compact set. These estimates start reflection and
exclude contact at infinity.

As in the compact-boundary proof of Theorem~\ref{index-one}, the comparison
principle and uniqueness of the centered expansion give \(\lambda_*=0\).
Reflection from both sides in every \(x'\)-direction gives
\(O(n-1)\)-invariance and \(-\nu_r>0\).

For fixed \((x_n,x_{n+1})\) with \(x_{n+1}>0\) and \(e\in S^{n-2}\),
reflection across \(e\cdot x'=\frac{r_1+r_2}{2}\) gives
\[
 \begin{gathered}
 (r_2e,x_n,x_{n+1})\in\mathcal E
 \quad\Longrightarrow\quad (r_1e,x_n,x_{n+1})\in\mathcal E,\quad 0\le r_1<r_2.
 \end{gathered}
\]
Thus each nonempty section of \(\mathcal E\) at fixed
\((x_n,x_{n+1})\) is a ball centered at \(0\) or all of \(\mathbb R^{n-1}\).

On \(D=M\cap\{x_1>0\}\), the positive Jacobi field
\(-\nu_1\) implies stability for functions vanishing on
\(M\cap\{x_1=0\}\). The rotation field
\(\psi=x_{n+1}\nu_1-x_1\nu_{n+1}\) satisfies
\[
 \Delta\psi+|h|^2\psi=0,\qquad
 \psi=-\cos\theta\,x_1>0\quad\text{on }\partial M\cap D,\qquad
 \psi=0\quad(x_1=0).
\]
On the ends, uniformly up to \(r=0\),
\[
 \frac{\psi}{x_1}=-\cos\theta+O(\rho^{1-n})>0,\qquad
 -\frac{\nu_r}{r}=O(\rho^{-n}).
\]
The negative part \(\psi_-\) has compact support and zero trace on both
boundary pieces. Testing gives \(Q(\psi_-,\psi_-)=0\), so stability makes
it a weak Jacobi field on \(D\). Every unbounded component meets an end;
the maximum principle for \(x_1\) makes every bounded component meet
\(\partial M\) where \(x_1>0\). Hence \(\psi_-\) vanishes on an open subset
of each component, and unique continuation gives \(\psi_-=0\).

By \(O(n-1)\)-invariance, \(\psi\ge0\) gives, for \(r>0\),
\[
 \nu_{n+1}\le \frac{x_{n+1}\nu_r}{r}<0,\qquad
 \Delta(\nu_{n+1}-\cos\theta)=-|h|^2\nu_{n+1}\ge0.
\]
The maximum principle gives \(\nu_{n+1}<0\) at \(r=0\) and, since \(M\) is
nonflat, \(\nu_{n+1}<\cos\theta\) in the interior.
Write \(\Pi_i(X)=X-x_iE_i\), \(1\le i\le n+1\).
Since \(M\) bounds \(\Omega\), successive intersections with a vertical
line would have opposite normal signs. Thus \(\Pi_{n+1}|_M\) is injective;
the end expansion makes it proper. The graph satisfies
\[
 |Dv|^2=\nu_{n+1}^{-2}-1<\tan^2\theta.
\]

On \(S=M\cap\{x_n=0\}\), symmetry gives \(\nu_n=0\). Thus zero is a
regular value of both \(x_n|_M\) and \(x_n|_{\partial M}\). Outside a compact set, the ends satisfy
\(|x_n|\ge \frac{b}{2}\); hence \(S\) is compact. At its boundary, \(\mu\in TS\) and
\(\mu_{n+1}=-\sin\theta\ne0\).

Since \(\Pi_{n+1}|_S\) is injective and \(S\) is rotationally symmetric, each
component is the graph of a radial function \(x_{n+1}=v(r)\), with
\[
 \begin{gathered}
 v>0\quad\text{in the interior},\qquad
 v=0\quad\text{on the boundary},\\
 v'(r)=\frac{\nu_r}{|\nu_{n+1}|}<0\quad(r>0).
 \end{gathered}
\]
Since \(S\) is compact, a component without boundary would have
a point where \(r\) is maximal. Its local graph representation
allows \(r\) to increase near that point, a contradiction.
If a component had an inner boundary at \(r=a>0\), then \(v(a)=0\).
At an interior point with \(r>a\),
\[
 v(r)=v(a)+\int_a^r v'(s)\,ds<0,
\]
contrary to \(v>0\).
Thus each component projects onto a ball centered at \(0\).
Two such components would both project to \(0\), contradicting
the injectivity of \(\Pi_{n+1}\). Finally, choose \(p_\pm\) on the two ends with
\(x_n(p_+)>0>x_n(p_-)\).
A path in \(M\) joining these points meets \(S\), so
\(S\ne\varnothing\).

Consequently, for some \(r_0>0\),
\[
 \begin{gathered}
 S=\{(x',0,v(|x'|)):|x'|\le r_0\},\qquad v(r_0)=0,\\
 K=\{(x',x_{n+1}):|x'|\le r_0,\ 0\le x_{n+1}\le v(|x'|)\}.
 \end{gathered}
\]
Set \(M_+=M\cap\{x_n\ge0\}\). By (i), \(\Pi_n\) identifies \(M_+\)
with \(\overline{(\mathbb R^{n-1}\times[0,\infty))\setminus K}\).
Since \(v'<0\), each ray from the origin meets \(\Pi_n(S)\) once.
Radial coordinates give
\[
 (M_+,S)\cong
 \bigl([0,\infty)\times\overline D^{n-1},
       \{0\}\times\overline D^{n-1}\bigr).
\]
Joining \(M_+\) and its reflected half along \(S\) gives
\(M\cong\mathbb R\times\overline D^{n-1}\) and
\(\partial M\cong\mathbb R\times S^{n-2}\).
Radial coordinates also give
\[
 \Sigma_M:=M\cap\{x_2=\cdots=x_{n-1}=0,\ x_1\ge0\}
 \cong\mathbb R\times[0,1],
\]
with boundary components at \(r=0\) and \(x_{n+1}=0\).

Hopf's lemma applied to \(\nu_{n+1}-\cos\theta\) gives
\[
 \partial_\mu\nu_{n+1}=-\frac{\sin^2\theta\,q}{\cos\theta}>0,\qquad q>0,\qquad
 H_{\partial M\subset\partial\mathbb R^{n+1}_+}=-\frac{q}{\cos\theta}>0.
\]
The map \(\Pi_{n+1}|_M\) has multiplicity one. The interval sections of
\(\Omega\) bound the multiplicity of \(\Pi_n|_M\) by two.
The radial ball sections of \(\{x_{n+1}>0\}\setminus\overline\Omega\) give
the same bound for \(\Pi_i|_M\),
\(1\le i\le n-1\).
Since \(1\le\sum_{i=1}^{n+1}|\nu_i|\), the area formula yields
\begin{equation}\label{universal-area}
 \Vol(M\cap B_R(p))\le(2n+1)\omega_nR^n,
 \qquad p\in\mathbb R^{n+1},\ R>0.
\end{equation}
\end{proof}

For the following identities, assume that \(M\) has two regular boundary
ends and is invariant under \(O(n-1)\) and reflection across \(x_n=0\),
with \(\nu_n>0\) when \(x_n>0\) and \(-\nu_r>0\) when \(r>0\).
No index assumption is needed.

On \(M_+=M\cap\{x_n>0\}\), the positive field \(\nu_n\) satisfies
\[
 \Delta\nu_n+|h|^2\nu_n=0,\qquad \partial_\mu\nu_n=q\nu_n,\qquad
 \nu_n=0\quad\text{on }S.
\]
For \(g\in C_c^\infty(M)\) odd in \(x_n\), integration by parts gives
\begin{equation}\label{odd-jacobi-identity}
 Q(g,g)=2\int_{M_+}\nu_n^2|\nabla(\frac{g}{\nu_n})|^2\,dA\ge0.
\end{equation}
Set \(a=-\nu_r>0\).
For \(f\in C_c^\infty(\Sigma_M\cap\{r>0\})\),
extend \(f\) by \(O(n-1)\)-symmetry to \(M\cap\{r>0\}\).
The next two identities use the induced metric on \(\Sigma_M\).
For a real spherical harmonic \(Y_k\) with
\(\|Y_k\|_{L^2(S^{n-2})}=1\) and
\(-\Delta_{S^{n-2}}Y_k=\lambda_kY_k\), \(\lambda_k=k(k+n-3)\),
write \(Q_k(f)=Q(fY_k,fY_k)\). Integration by parts gives
\begin{equation}\label{angular}
 Q_k(f)=\int_{\Sigma_M}
 r^{n-2}\left[
 a^2|\nabla(\frac{f}{a})|^2
 +(\lambda_k-n+2)\frac{f^2}{r^2}
 \right]dA,\qquad k\ge1.
\end{equation}
Indeed, separation of variables gives
\[
 Q_k(f)=\int_{\Sigma_M}
 r^{n-2}\bigl(|\nabla f|^2+(\lambda_kr^{-2}-|h|^2)f^2\bigr)\,dA
 -\int_{\partial M\cap\Sigma_M}
 r^{n-2}qf^2\,ds.
\]
The equations for the translation fields \(aY_1\) give \eqref{angular}.
If \(fY_k\) is smooth on \(M\), then \(f=O(r^k)\) near \(r=0\).
The Dirichlet energy of \(fY_k\) in \(r<2\epsilon\) is
\(O(\epsilon^{n+2k-3})\) on compact sets.
Thus \eqref{angular} extends across \(r=0\).

Extend \(a=-\nu_r\) by zero at \(r=0\). The translation equations give
\[
 \Delta a+|h|^2a=\frac{(n-2)a}{r^2},\qquad \partial_\mu a=qa.
\]
Near compact subsets of \(\{r=0\}\) and on the regular ends, respectively,
\[
 a=O(r),\quad |Da|=O(1),\qquad
 a=O(\rho^{1-n}),\quad |Da|=O(\rho^{-n}).
\]
Hence \(a\in\mathcal D^{1,2}(M)\). Integration outside \(r<\epsilon\),
followed by approximation at infinity, gives
\begin{equation}\label{radial-negative}
 Q(a,a)=-(n-2)\int_M\frac{a^2}{r^2}\,dA<0.
\end{equation}
The boundary term at \(r=\epsilon\) is \(O(\epsilon^{n-1})\) on compact sets, and the end estimates
make the outer errors vanish. Choose \(f_j\in C_c^\infty(M)\) with
\(f_j\to a\) in \(\mathcal D^{1,2}(M)\). Continuity of \(Q\) gives
\[
 Q(f_j,f_j)\longrightarrow Q(a,a)<0.
\]
Hence \(Q(f_j,f_j)<0\) for all sufficiently large \(j\).

\begin{prop}\label{kernel}
Under the hypotheses of Proposition~\ref{two-ends}, if \(\Ind_\theta(M)=1\), then
\[
 \ker_{\mathcal D^{1,2}(M)}(\Delta+|h|^2,\partial_\mu-q)
 =\operatorname{span}\{\nu_1,\ldots,\nu_{n-1}\}.
\]
\end{prop}
\begin{proof}
Let \(F\) belong to this kernel. For any such field \(G\),
choose a smooth radial function \(\chi_R=1\) on \(B_R\) and
\(\chi_R=0\) outside \(B_{2R}\), with \(|d\chi_R|\le \frac{C}{R}\).
By \eqref{Sob} and \eqref{universal-area},
\[
 Q(\chi_RG,\chi_RG)=\int_M G^2|d\chi_R|^2
 \le C\|G\|_{L^{\frac{2n}{n-2}}(M\setminus B_R)}^2\longrightarrow0.
\]
Let \(F_{\mathrm o}\) be the part of \(F\) odd in \(x_n\).
Applying \eqref{odd-jacobi-identity} to \(\chi_RF_{\mathrm o}\) and
using Fatou's lemma gives \(\nabla(\frac{F_{\mathrm o}}{\nu_n})=0\) on
\(M\cap\{x_n>0\}\).
Thus \(F_{\mathrm o}=c\nu_n\). Since \(\nu_n\to\pm\sin\theta\) on the ends,
\(\nu_n\notin L^{\frac{2n}{n-2}}(M)\), so \eqref{Sob} gives \(c=0\).

Choose an orthonormal basis \(Y_{k,\ell}\) of real degree-\(k\)
spherical harmonics on \(S^{n-2}\) and write
\(F=\sum_{k,\ell}f_{k,\ell}Y_{k,\ell}\).
Each term is a Robin Jacobi field in \(\mathcal D^{1,2}(M)\).
Applying \eqref{angular} to the cutoffs for \(k\ge1\) and using Fatou's lemma gives
\[
 f_{k,\ell}=0\quad(k\ge2),\qquad
 f_{1,\ell}=c_\ell(-\nu_r).
\]
The \(k=1\) terms therefore lie in
\(\operatorname{span}\{\nu_1,\ldots,\nu_{n-1}\}\), whose elements have
finite energy by the end estimates. Subtracting them leaves an
\(O(n-1)\)-invariant field \(f\), even in \(x_n\). Suppose that \(f\ne0\).

Apply the index-one orthogonality argument of
Section~\ref{sec:classification-three} to the restrictions of \(f\)
to the connected components of \(\{f\ne0\}\), extended by zero.
Unique continuation \cite{Aro57} then gives at most two such components.

On each end \(E\), write
\[
 A_{s,t}=\{p\in E:s<\rho(p)<t\},\qquad \gamma=\frac{n-2}{2}.
\]
Sobolev and annular Robin estimates give
\[
 \sup_{A_{R,2R}}|D^kf|
 \le C_kR^{-k-\gamma}\|f\|_{L^{\frac{2n}{n-2}}(A_{\frac{R}{2},4R})}
 =o(R^{-k-\gamma}).
\]
In the coordinate \(\xi=\frac{x_{n+1}}{\sin\theta}\), the height variation
\(v=\frac{f}{\nu_n}\) satisfies \(Lv=0\) and \(B_0v=0\), where
\begin{equation}\label{jacobi-end-operators}
 \begin{aligned}
 L&=\Delta+O_k(\rho^{1-n})D^2+O_k(\rho^{-n})D,\\
 B_0&=\partial_\xi+O_k(\rho^{1-n})D_{x'}.
 \end{aligned}
\end{equation}
Thus \(\Delta v=o(\rho^{-n-\gamma-1})\) and
\(v_\xi=o(\rho^{-n-\gamma})\) on \(\xi=0\).
For \(v\), define \(\lambda\) as in \eqref{reflected-end-source}.
The preceding bounds give
\(\int_{\mathbb R^n}(1+|\zeta|)\,d|\lambda|(\zeta)<\infty\), since \(\gamma>0\).
The Newtonian potential argument of Section~\ref{sec:structure}, rescaled
estimates, and evenness in \(x_n\) give, on both ends,
\[
 f=A_0\rho^{2-n}+O_k(\rho^{1-n}).
\]
The dilation Jacobi field satisfies
\[
 Z=X\cdot\nu
 =b\sin\theta-(n-1)m\sin\theta\,\rho^{2-n}+O_1(\rho^{1-n}).
\]
Green's identity, with cancellation of the capillary boundary terms, yields
\[
 \begin{aligned}
 0&=\lim_{R\to\infty}\sum_\pm\int_{\rho=R}
       (Z\partial_\eta f-f\partial_\eta Z)\\
  &=(2-n)|S^{n-1}|b\sin\theta\,A_0.
 \end{aligned}
\]
Thus \(A_0=0\), and \(f=O_k(\rho^{1-n})=o(\rho^{2-n})\).

If \(v\) had one sign near infinity on an end, the maximum principle
would give \(v>0\) after a sign change. The barrier argument in
Proposition~\ref{two-ends} gives
\[
 v\ge\epsilon H\ge\epsilon\rho^{2-n},
\]
contradicting
\(v=o(\rho^{2-n})\). Thus \(f\) takes both signs outside every compact subset
of each end.

Near a point with \(r=0\) and \(x_n>0\), we have \(\nu_n>0\), and analyticity gives
\(v(r,x_{n+1})=\sum_{j\ge0}a_j(x_{n+1})r^{2j}\). In the \(r^{2j}\) equation,
\(a_{j+1}\) has coefficient \(2(j+1)(2j+n-1)\) times the positive transverse
principal coefficient; all other terms involve \(a_0,\ldots,a_j\) and
their derivatives. Hence
\[
 a_0=\cdots=a_j=0\quad\Longrightarrow\quad a_{j+1}=0.
\]
If \(f\) vanished on \(\{r=0\}\), then \(a_0=0\), and induction would give
\(a_j=0\) for every \(j\ge0\). Analyticity and unique continuation would
then give \(f=0\), a contradiction.

Since \(-\nu_{n+1}>0\), with limit \(-\cos\theta>0\) on both ends,
\[
 \operatorname{div}_M\!\left(\nu_{n+1}^2\nabla\frac{f}{-\nu_{n+1}}\right)=0,
 \qquad \frac{f}{-\nu_{n+1}}\longrightarrow0.
\]
On each connected component \(D\) of \(\{f\ne0\}\), the maximum principle gives
\[
 \max_{\overline D}\frac{|f|}{-\nu_{n+1}}
 =\max_{\overline D\cap\partial M}\frac{|f|}{-\nu_{n+1}}>0.
\]
Since both signs occur and there are at most two components,
\(\{f>0\}\) and \(\{f<0\}\) are connected and each meets \(\partial M\).

Choose the sign of \(f\) so that \(\{f>0\}\) meets \(\{r=0\}\).
By Proposition~\ref{two-ends}, \(\Sigma_M\cong\mathbb R\times[0,1]\).
The continuous map
\((x',x_n,x_{n+1})\mapsto(|x'|,0,\ldots,0,x_n,x_{n+1})\)
preserves \(f\), so \(\Sigma_M\cap\{f>0\}\) and
\(\Sigma_M\cap\{f<0\}\) are connected.
A compact simple curve in \(\Sigma_M\cap\{f>0\}\)
joins \(r=0\) to \(\partial M\) and separates the two ends.
Since \(\Sigma_M\cap\{f<0\}\) meets both ends, it must cross this curve,
a contradiction.
\end{proof}

Normalize the radius of \(\partial M\cap\{x_n=0\}\) to one and choose
\(\cos\theta<0\). Fix \(\frac{\pi}{2}<\theta_0<\theta_1<\pi\).
The constants in Lemma~\ref{uniform-estimates}
are uniform for \(\theta\in[\theta_0,\theta_1]\), in every dimension
\(n\ge3\). When \(3\le n\le6\), we may also take \(\theta_0=\frac{\pi}{2}\):
the estimates then hold for \(\theta\in(\frac{\pi}{2},\theta_1]\) and include
the limit \(\theta\downarrow\frac{\pi}{2}\).
The upper half is the graph \(x_n=u(r,x_{n+1})>0\)
over the exterior of the compact region \(K\) bounded by \(S\) and the support.
Write \(S\) as \(x_{n+1}=v(r)\), \(0\le r\le1\), with
\[
 v(1)=0,\qquad v'<0\ (0<r<1),\qquad T=v(0).
\]

\begin{lemm}\label{uniform-estimates}
There are uniform positive constants such that
\[
 \begin{gathered}
 r|h|\le C\quad(r>0),\qquad 0<T_-\le T\le1,\\
 b_0\le b\le b_1,\qquad m_0\le m\le m_1.
 \end{gathered}
\]
Curvature and all its derivatives are uniformly bounded on compact
subsets of the ambient space. For \(\rho\) beyond a uniform radius,
the upper end satisfies
\begin{equation}\label{uniform-decay}
 0<V=b+\kappa x_{n+1}-u\le C\rho^{2-n},\qquad
 |D^jV|\le C_j\rho^{2-n-j}.
\end{equation}
\end{lemm}
\begin{proof}
Choose coordinates in the \(x'\)-variables so that
\(x'(p)=r(p)E_1\). As in the proof of
Proposition~\ref{two-ends}(ii), the positive field \(-\nu_1\) gives
stability on \(M\cap B_{\frac{r(p)}{4}}(p)\).

Suppose that \(r|h|\) is not uniformly bounded. Choose \(M_j,p_j\) with
\(r_j|h(p_j)|\to\infty\), \(r_j=r(p_j)\), and
\(\theta_j\to\theta_*\in[\theta_0,\theta_1]\). Maximize
\[
 (\frac{r_j}{4}-|x-p_j|)|h(x)|
\]
on the compact set \(M_j\cap\overline B_{\frac{r_j}{4}}(p_j)\), at \(p_j^*\).
Writing \(d_j=\frac{r_j}{4}-|p_j^*-p_j|\), \(\lambda_j=|h(p_j^*)|\), we obtain
\[
 \lambda_jd_j\to\infty,\quad
 |h|\le2\lambda_j\quad\text{on }B_{\frac{d_j}{2}}(p_j^*),\quad
 r(p_j^*)\ge\frac{3r_j}{4}.
\]
The rescaled hypersurfaces \(M_j'=\lambda_j(M_j-p_j^*)\) are stable and satisfy
\[
 |h|(0)=1,\qquad |h|\le2,\qquad
 \Vol(M_j'\cap B_R)\le(2n+1)\omega_nR^n
\]
on balls whose radii tend to infinity. Also \(\lambda_jr(p_j^*)\to\infty\).

In local parametrizations, the curvature, area, and boundary estimates
give a complete, two-sided minimal immersion \(M_\infty\)
and a subsequence with
\[
 M_j'\longrightarrow M_\infty
 \quad\text{smoothly up to the boundary},\qquad
 |h_\infty|(0)=1.
\]
The rescaled boundaries satisfy
\[
 \partial M_j'\subset H_j
 :=\{x_{n+1}=-\lambda_jx_{n+1}(p_j^*)\}.
\]
After passing to a subsequence, either
\(\operatorname{dist}(0,H_j)\to\infty\) or \(H_j\to H_*\).

Choose coordinates so that \(x'(p_j^*)=r(p_j^*)E_1\), and set
\(f_j=\frac{-\nu_{1,j}}{-\nu_{1,j}(0)}\).
Interior and boundary Harnack estimates \cite{Chicco05} give
\[
 \begin{gathered}
 f_j\longrightarrow f_\infty>0,\qquad f_\infty(0)=1,\\
 (\Delta+|h_\infty|^2)f_\infty=0,\qquad
 \partial_\mu f_\infty=q_\infty f_\infty
 \quad\text{on }\partial M_\infty,
 \end{gathered}
\]
where \(q_\infty=\cot\theta_*\,h_\infty(\mu,\mu)\).
Thus \(M_\infty\) is stable; see \cite[Section 5]{HS23}.

The \(O(n-1)\)-symmetry gives, for every fixed \(R\),
\[
 \sup_{M_j'\cap B_R}|\nu_{i,j}|
 \le\frac{R}{\lambda_jr(p_j^*)-R}\longrightarrow0,
 \qquad 2\le i\le n-1.
\]
Hence \(E_2,\ldots,E_{n-1}\) are tangent to \(M_\infty\)
and its boundary. Completeness gives
\[
 M_\infty=\Sigma^2\times\mathbb R^{n-2}.
\]
Here \(\Sigma\subset\mathbb R^3\) is connected, complete,
two-sided, and minimal.
If \(H_j\to H_*\), the contact angle along
\(\partial M_\infty\) is \(\theta_*\).
Interior contact with \(H_*\) would imply \(h_\infty=0\)
by the height maximum principle, contradicting \(|h_\infty|(0)=1\).

The surface \(\Sigma\) is stable and properly immersed, with
\(\operatorname{Area}(\Sigma\cap B_R)\le CR^2\).
If \(\partial\Sigma=\varnothing\), the stable Bernstein theorem
\cite{dCP79,FCS80,Pog81} gives a plane.
Otherwise, the capillary stable Bernstein theorems
\cite[Theorem 1.5]{HS23}, \cite[Theorem 2.10]{LZZ25} give a half-plane.
Thus \(h_\infty=0\), contradicting \(|h_\infty|(0)=1\).
Hence \(r|h|\le C\), so \(|h|\le2C\) for \(r\ge\frac12\).

Suppose that \(\theta_0>\frac{\pi}{2}\), and set
\(\kappa_0=-\cot\theta_0>0\).
By Proposition~\ref{two-ends}(ii), \(M\) is the graph
\(x_{n+1}=v(x',x_n)\), with
\(|Dv|\le\frac{1}{\kappa_0}\) and
\(\partial\{v>0\}\subset\{r\ge1\}\).
Since \(v(0,0)>0\) and the cylinder
\(\{r<1\}\times\mathbb R\) is connected,
\[
 \{r<1\}\times\mathbb R\subset\{v>0\}.
\]
Interior graph estimates therefore apply on base balls of radius
\(\frac14\) when \(r<\frac12\).

If \(\theta_0=\frac{\pi}{2}\) and \(n\le6\), then
\(r(p)<\frac12\) implies
\(\operatorname{dist}(p,\partial M)\ge\frac12\).
The positive Jacobi field \(-\nu_{n+1}\) gives stability on
\(M\cap B_{\frac14}(p)\).
The interior curvature estimate \cite[Theorem 2]{Bel25}
and local graph estimates, together with \eqref{universal-area},
give the derivative bounds.
Thus, in both cases,
\[
 \sup_{M\cap\{r<\frac12\}}|\nabla^j h|\le C_j
 \qquad(j\ge0).
\]

On \(S\), the rotation inequality \(x_{n+1}\nu_r-r\nu_{n+1}\ge0\) gives
\[
 r+vv'\ge0,\qquad (r^2+v^2)'\ge0,\qquad r^2+v^2\le1.
\]
Hence \(T\le1\) and \(K\subset B_1\). Let \(\gamma\) be the unit-speed geodesic in \(M\) starting at
\(\partial S\), with
\[
 \gamma'(0)=\cos\theta E_r+\sin\theta E_{n+1}.
\]
Reflection symmetry gives \(\gamma\subset S\).
The curvature bound gives \(|\gamma''|\le C\) near \(\partial S\).
Choose \(0<\ell_0\le\frac14\) with
\(C\ell_0\le\frac{\sin\theta_1}{2}\).
Then
\(\gamma_{n+1}'(\ell)\ge\sin\theta-C\ell
\ge\frac{\sin\theta_1}{2}\) for \(0\le\ell\le\ell_0\), so
\[
 T\ge\gamma_{n+1}(\ell_0)
 =\int_0^{\ell_0}\gamma_{n+1}'(\ell)\,d\ell
 \ge\frac{\ell_0\sin\theta_1}{2}=:T_->0.
\]

For \(G=\frac{Du}{\sqrt{1+|Du|^2}}+\cos\theta E_{n+1}\),
\[
 \operatorname{div}G=0,\qquad
 G\cdot E_{n+1}=0\quad(x_{n+1}=0).
\]
Let \(n_K\) be the outward unit normal to \(K\) along \(S\).
As the graph approaches \(S\), \(\frac{Du}{\sqrt{1+|Du|^2}}\to n_K\).
Apply the divergence theorem to \(\{u>\epsilon\}\cap B_R\).
As \(\epsilon\downarrow0\), the level sets
\(\{u=\epsilon\}\) converge smoothly to \(S\), and
\[
 \int_{\{u=\epsilon\}}G\cdot\frac{Du}{|Du|}\,dA
 \longrightarrow
 \int_S(1+\cos\theta\,n_K\cdot E_{n+1})\,dA
 =|S|+\omega_{n-1}\cos\theta.
\]
Using \eqref{main-expansion} with
\(x_{n+1}=\xi\sin\theta\) and letting \(R\to\infty\) gives
\[
 \frac{(n-2)|S^{n-1}|}{2}\sin^2\theta\,m
 =|S|+\omega_{n-1}\cos\theta.
\] Also
\[
 \begin{split}
 \omega_{n-1}
 &\le |S|=|S^{n-2}|\int_0^1r^{n-2}\sqrt{1+v'^2}\,dr\\
 &\le\omega_{n-1}+|S^{n-2}|T\le n\omega_{n-1}.
 \end{split}
\]
Since \(\theta\in[\frac{\pi}{2},\theta_1]\), these inequalities give uniform
positive lower and upper bounds for \(m\).

On the upper boundary, write \(x_n=g(r)=u(r,0)\) and
\(p=\frac{g'}{\sqrt{1+g'^2}}>0\). The positive mean curvature of this boundary gives
\[
 -p'-\frac{(n-2)p}{r}>0,\qquad p(1+)=1,\qquad
 p(r)\le r^{2-n},\qquad 0<g(r)\le\operatorname{arcosh}r.
\]
The last inequality follows by integrating
\(g'(r)\le(r^{2n-4}-1)^{-\frac{1}{2}}\le(r^2-1)^{-\frac{1}{2}}\).
With \(a=\frac{u_r}{\sqrt{1+|Du|^2}}=-\nu_r>0\), apply the divergence theorem
to \(G\) on \(\{r<R,0<x_{n+1}<H\}\setminus K\).
By \eqref{main-expansion}, the boundary integral over
\(\{r<R,x_{n+1}=H\}\) tends to zero as \(H\to\infty\), giving
\begin{equation}\label{cylinder}
 |S^{n-2}|R^{n-2}\int_0^\infty a(R,x_{n+1})\,dx_{n+1}
 =\frac{(n-2)|S^{n-1}|}{2}\sin^2\theta\,m.
\end{equation}
For \(R\ge2\),
\[
 \nu_n(R,0)=\sin\theta\sqrt{1-p(R)^2}
                    \ge\frac{\sqrt3\sin\theta}{2}.
\]
The curvature bound above gives, for a uniform
\(\delta>0\), a graph on a base ball of radius \(\delta R\) with
\(\nu_n\ge\frac{\sin\theta}{2}\). The Harnack inequality up to the boundary
\cite{Chicco05} gives
\[
 a(R,x_{n+1})\ge\gamma_0a(R,0)=\gamma_0\sin\theta\,p(R)
       \quad(0\le x_{n+1}\le\delta_0R).
\]
Here \(\gamma_0\) and \(\delta_0\) are uniform.
Equation~\eqref{cylinder} now implies
\[
 p(R)\le CmR^{1-n},\qquad g'(R)\le CmR^{1-n},
\]
and hence
\[
 b\le\operatorname{arcosh}2+Cm,\qquad
 0<b-g(R)\le CmR^{2-n}\quad(R\ge2).
\]
The boundary curve in the \((r,x_n)\)-plane starts in direction \(E_n\).
The curvature bound gives \(\frac{dx_n}{d\ell}\ge\frac{1}{2}\) on an initial segment
of uniform positive length. Thus \(b\ge b_0>0\).

Proposition~\ref{two-ends} gives
\[
 u_r>0,\qquad
 u_{x_{n+1}}\ge\kappa\sqrt{1+u_r^2}\ge\kappa,\qquad
 V_r<0,\quad V_{x_{n+1}}\le0.
\]
Consequently,
\[
 \begin{gathered}
 V(r,x_{n+1})\le Cr^{2-n}\quad(r\ge2),\qquad V\le b\quad(r>1),\\
 V\le b+\kappa v(r)\le b+\kappa\quad(r\le1).
 \end{gathered}
\]
For \(r\ge2\) and \(0\le x_{n+1}\le\delta r\), boundary graph estimates
give a uniform bound \(|Du|\le C\). For large \(\rho\) and
\(x_{n+1}\ge\delta r\), the interior gradient estimate on balls of radius
\(\frac{x_{n+1}}{2}\), together with the preceding bound for \(V\), gives the same bound.

Subtracting the affine solution yields
\[
 \operatorname{div}(\bar A DV)=0,\qquad
 (\bar A DV)\cdot E_{n+1}=0\quad(x_{n+1}=0),
\]
where \(\bar A\) is the uniformly elliptic matrix
\[
 \bar A=\int_0^1
 \left.D_p\!\left(\frac p{\sqrt{1+|p|^2}}\right)
       \right|_{p=\kappa E_{n+1}-tDV}\,dt.
\]
For \(r<x_{n+1}\), a fixed interior Harnack chain joins
\((r,x_{n+1})\) to \((x_{n+1},x_{n+1})\), giving
\(V\le Cx_{n+1}^{2-n}\). For \(r\ge x_{n+1}\), use
\(V\le Cr^{2-n}\). Thus \(V\le C\rho^{2-n}\).
Boundary regularity after rescaling bounds the H\"older norms of \(\bar A\)
independently of the annulus radius and the normalized hypersurface.
Schauder estimates then give \eqref{uniform-decay} for every \(j\).
\end{proof}

In fixed coordinates on the ends, define the weighted H\"older norm by
\[
 \|f\|_{C^{k,\alpha}_{-\delta}}
 =\|f\|_{C^{k,\alpha}(M\cap B_{R_0})}
  +\sup_{R\ge R_0}R^\delta
        \|f(R\,\cdot)\|_{C^{k,\alpha}(A_{1,2})}.
\]
The supremum is taken over all ends.

\begin{prop}\label{properness}
Let \(M_j\) be index-one hypersurfaces satisfying the assumptions of this
appendix, with the coordinates and symmetries of
Proposition~\ref{two-ends}, normalized so that
\(\partial M_j\cap\{x_n=0\}\) has radius one. If
\(\theta_j\to\theta_*\in[\theta_0,\theta_1]\), then a subsequence converges smoothly
with multiplicity one to a connected embedded hypersurface with two regular boundary
ends and index one. After subtracting the corresponding affine end
functions, convergence also holds
in \(C^{k,\alpha}_{-\delta}\) for each fixed \(k\), \(0<\alpha<1\), and
\(0<\delta<n-2\). The same holds with the normalization \(b=1\).
\end{prop}
\begin{proof}
Lemma~\ref{uniform-estimates} gives uniform curvature
bounds and \(K_j\subset B_1\). Equation~\eqref{wedge} bounds the upper
graphs by \(0<u_{j,+}<b_1-x_{n+1}\cot\theta_1\). Local smooth compactness
gives subsequential convergence to \(M_*\) on each local component.

Parametrize \(S_j\cap\Sigma_{M_j}\) by arclength:
\[
 \begin{gathered}
 \gamma_j(s)=(r_j(s),t_j(s)),\quad 0\le s\le L_j,\qquad
 r_j'\le0,\quad t_j'\ge0,\quad |\gamma_j'|=1,\\
 1\le L_j\le1+T_j\le2,\qquad |t_j'|\le Cr_j.
 \end{gathered}
\]
After taking a subsequence, \(L_j\to L\) and \(\gamma_j\to\gamma=(r,t)\)
smoothly, with
\[
 \gamma(0)=(1,0),\qquad \gamma(L)=(0,T_*),\qquad T_*\ge T_-.
\]
Monotonicity and \(|\gamma'|=1\) make \(\gamma\) injective;
\(|t'|\le Cr\) excludes an interval contained in \(\{r=0\}\).
Rotating \(\gamma\) therefore gives a single smooth disk \(S_*\).

On the boundary sphere,
\[
 -\nu_{r,j}=\sin\theta_j\ge\sin\theta_1>0.
\]
The radial graph converges there with multiplicity one. Let \(M_*^0\)
be the limiting component containing this sphere; the convergence of
\(\gamma_j\) gives \(S_*\subset M_*^0\).
Write \(g_j(r)=u_{j,+}(r,0)\). The boundary curves
\[
 \{(r,0,\ldots,0,\pm g_j(r),0):1\le r\le R\}
\]
are parametrized by arclength, starting at \(E_1\). Since
\(g_j'\ge0\) and \(0<g_j<b_1\), their lengths satisfy
\[
 \ell_{j,R}^{\pm}
 =\int_1^R\sqrt{1+g_j'(r)^2}\,dr
 \le R-1+g_j(R)\le R-1+b_1.
\]
For each fixed \(R\), the boundary estimates give convergent curves
in \(M_*^0\). Choose \(R_e>2\) beyond the uniform radius in
Lemma~\ref{uniform-estimates}, with \(CR_e^{2-n}<\frac{b_0}{2}\), and set
\[
 E_{j,\pm}
 =\{(x',u_{j,\pm}(x',t),t):t\ge0,\ |x'|^2+t^2>R_e^2\}.
\]
Their limiting graphs \(E_{*,\pm}\) are connected. Taking \(R>R_e\),
the two limiting boundary curves join the sphere to \(E_{*,+}\) and
\(E_{*,-}\), respectively. Thus
\[
 S_*\cup E_{*,+}\cup E_{*,-}\subset M_*^0.
\]
Let \(\kappa_1=-\cot\theta_1\). Equation~\eqref{wedge} gives
\[
 |X|^2\le R_e^2+(b_1+\kappa_1R_e)^2
 \quad\text{when }|x'|^2+x_{n+1}^2\le R_e^2.
\]
For \(R_0>\sqrt{R_e^2+(b_1+\kappa_1R_e)^2}\), every other limiting
component \(N\ne M_*^0\) therefore satisfies
\[
 M_j\setminus B_{R_0}\subset E_{j,+}\cup E_{j,-},\qquad
 N\subset\overline B_{R_0}\quad(N\ne M_*^0).
\]
By \eqref{universal-area}, \(N\) has finite volume. The intrinsic volume
lower bound following \eqref{Sob} therefore makes \(N\) compact.
If \(\partial N\ne\varnothing\), the harmonic height gives
\(x_{n+1}\equiv0\), contradicting \(\sin\theta_*>0\).
A closed component contradicts
\(\int_N\Delta|X|^2=2n\Vol(N)>0\).

The limiting normal satisfies \(\nu_r,\nu_{n+1}\le0\). Its values on the
sphere and \(\{r=0\}\), together with the strong maximum principle, give
\[
 \nu_r<0\quad(r>0),\qquad \nu_{n+1}<0\quad\text{in the interior}.
\]
The local graphs \(r=r(x_n,x_{n+1})\) for \(r>0\) and
\(x_{n+1}=v(x',x_n)\) in the interior exclude sheet contact. Multiplicity one at
the sphere extends over the connected limit.

After passing to a subsequence, \(b_j\to b_*>0\) and \(m_j\to m_*>0\).
The estimates in Lemma~\ref{uniform-estimates} pass to the limit, giving
two regular ends and finite total curvature.
Subtracting the affine end functions gives
\[
 \|u_j-u_*-(b_j-b_*)-(\kappa_j-\kappa_*)x_{n+1}\|
       _{C^{k,\alpha}_{-\delta}(\{\rho>R\})}
 \le CR^{\delta+2-n}\qquad(0<\delta<n-2).
\]
Together with smooth convergence on compact sets, this gives weighted convergence.

Let \(Q_j\) and \(Q_*\) denote the quadratic forms on \(M_j\) and \(M_*\).
If \(\Ind_{\theta_*}(M_*)\ge2\), choose
\(f_1,f_2\in C_c^\infty(M_*)\) with
\(Q_*(f_\alpha,f_\beta)=-\delta_{\alpha\beta}\).
Smooth convergence gives compactly supported functions
\(f_{\alpha,j}\) on \(M_j\) with
\[
 \bigl(Q_j(f_{\alpha,j},f_{\beta,j})\bigr)_{\alpha,\beta=1}^{2}
 \longrightarrow-I_2,
\]
contradicting \(\Ind_{\theta_j}(M_j)=1\).
Equation~\eqref{radial-negative} gives
\(\Ind_{\theta_*}(M_*)\ge1\), hence
\(\Ind_{\theta_*}(M_*)=1\).
Finally, \(b_j\to b_*>0\) gives
\(b_j^{-1}M_j\to b_*^{-1}M_*\), and these rescaled hypersurfaces
have \(b=1\).

\end{proof}

Use the coordinates of Proposition~\ref{two-ends} and
normalize by \(b=1\), so that the end planes are
\(x_n=\pm(1+\kappa x_{n+1})\).
Write \(\kappa=-\cot\theta\ge0\).
Even reflection, the catenoid index theorem~\cite[Theorem 2.1]{TZ09},
and the Jacobi-field analysis of~\cite[Section 2.2]{KP07} give
\(\Ind_{\frac{\pi}{2}}(C^+)=1\) and
\begin{equation}\label{half-catenoid-kernel}
 \ker_{\mathcal D^{1,2}(C^+)}(\Delta+|h|^2,\partial_\mu)
 =\operatorname{span}\{\nu_1,\ldots,\nu_{n-1}\}.
\end{equation}

Fix \(\frac{n-2}{2}<\delta<n-2\). In this proof, all functions and
boundary data are \(O(n-1)\)-symmetric and even in \(x_n\).

\begin{prop}\label{angle-chart}
For every \(n\ge3\), near each normalized index-one hypersurface with
\(\kappa>0\) and the regular ends assumed in this appendix, the normalized
solutions form a unique smooth curve parametrized by the contact angle. Near \(C^+\),
this curve is parametrized by a half-interval \(\kappa\ge0\).
These hypersurfaces are complete, properly embedded, have two regular boundary
ends, and have index one.
\end{prop}
\begin{proof}
Let \(M_0=C^+\).
For \(F\in C^{0,\alpha}_{-\delta-2}(M_0)\) and
\(G\in C^{1,\alpha}_{-\delta-1}(\partial M_0)\), consider
\[
 (\Delta+|h_0|^2)f=F\quad\text{on }M_0,\qquad
 (\partial_{\mu_0}-q_0)f=G\quad\text{on }\partial M_0.
\]
For \(O(n-1)\)-symmetric functions even in \(x_n\),
\(Q\) is represented by \(I-K\) on
\(\mathcal D^{1,2}(M_0)\), with \(K\) compact,
and \eqref{half-catenoid-kernel} gives
\(\ker(I-K)=\{0\}\).
Since \(\delta>\frac{n-2}{2}\), we have
\(F\in L^{\frac{2n}{n+2}}(M_0)\) and
\(G\in L^{\frac{2(n-1)}{n}}(\partial M_0)\).
Hence there is a unique weak solution
\(f\in\mathcal D^{1,2}(M_0)\).
The annular Robin estimates used in Proposition~\ref{kernel} give
\(f=o(\rho^{-\frac{n-2}{2}})+O(\rho^{-\delta})\to0\).

For the end operators \eqref{jacobi-end-operators}, choose
\(\epsilon_\delta>0\) sufficiently small and set
\(H_\delta=\rho^{-\delta}
-\epsilon_\delta\xi\rho^{-\delta-1}>0\).
For large \(\rho\) and some \(\gamma>0\),
\[
 LH_\delta\le-\gamma\rho^{-\delta-2},\qquad
 B_0H_\delta\le-\gamma\rho^{-\delta-1}.
\]
The comparison in Proposition~\ref{two-ends}, applied to
\(\frac{f}{\nu_{0,n}}\), gives \(f=O(\rho^{-\delta})\).
Schauder estimates \cite[Theorem 7.3]{ADN59} and the variational bound yield
\[
 \|f\|_{C^{2,\alpha}_{-\delta}(M_0)}
 \le C\left(
 \|F\|_{C^{0,\alpha}_{-\delta-2}(M_0)}
 +\|G\|_{C^{1,\alpha}_{-\delta-1}(\partial M_0)}
 \right).
\]
Since \(C^{2,\alpha}_{-\delta}(M_0)
\subset\mathcal D^{1,2}(M_0)\), uniqueness also holds
in the weighted space.

Choose a smooth odd function \(\chi(x_n)\) equal to \(\pm1\) on the two ends outside
a compact set. For \(\epsilon=\kappa\ge0\), define
\[
 X_\epsilon=X+\epsilon x_{n+1}\chi(x_n)E_n.
\]
This deformation preserves the support and symmetries. Its end planes are
\(x_n=\pm(1+\epsilon x_{n+1})\), and the change in the differential
is \(O(|\epsilon|)\). Choose a transverse field \(T\) equal to
\[
 \frac{E_n}{\nu_{0,n}}\quad\text{on the ends},
 \nu_0 \text{on the compact part of the boundary}.
\]
Combine these fields smoothly, keeping \(\langle T,\nu_0\rangle=1\) and
\(T\) tangent to the support on \(\partial M_0\).
For \(v\in C^{2,\alpha}_{-\delta}(M_0)\), use the graph
\(X_{\epsilon,v}(p)=X_\epsilon(p)+v(p)T(p)\).

Write the mean curvature and contact angle equations for
\(X_{\epsilon,v}\) as \(\mathcal F(\epsilon,v)=0\), where
\(\cos\theta=-\frac{\epsilon}{\sqrt{1+\epsilon^2}}\).
The map \(\mathcal F\) is smooth, and
\[
 D_v\mathcal F(0,0)
 =(\Delta+|h_0|^2,\partial_{\mu_0}-q_0),
\]
which has the bounded inverse proved above.
For \(v=0\), the mean curvature and the upper boundary error satisfy
\[
 H(X_\epsilon)=O(\epsilon\rho^{-n}),\qquad
 \epsilon\bigl(1-\sqrt{1+|D_{x'}u|^2}\bigr)
 =O(\epsilon\rho^{2-2n}).
\]
Since \(\delta<n-2\), we have
\(\mathcal F(\epsilon,0)\in
C^{0,\alpha}_{-\delta-2}(M_0)\times
C^{1,\alpha}_{-\delta-1}(\partial M_0)\).
The implicit function theorem gives a unique smooth solution
\(v(\epsilon)\) for sufficiently small \(\epsilon\ge0\), with
\[
 v(0)=0,\qquad
 \mathcal F(\epsilon,v(\epsilon))=0,\qquad
 \|v(\epsilon)\|_{C^{2,\alpha}_{-\delta}(M_0)}
 \le C\epsilon.
\]

The two ends satisfy
\(x_n^\pm=\pm(1+\epsilon x_{n+1})+O(\rho^{-\delta})\).
For large \(\rho\), they are disjoint and avoid the compact part.
On the remaining compact set, embeddedness follows from
\(C^1\)-closeness to \(M_0\).
The end graphs are proper, and the induced metrics are uniformly
equivalent to that of \(M_0\), giving completeness.
The symmetries preserve the center, and
\(v=O(\rho^{-\delta})\) preserves \(b=1\).

On the upper end, write
\(x_n=1+\kappa x_{n+1}+w\), with
\(w\in C^{2,\alpha}_{-\delta}\).
In the coordinates of \eqref{end-PDE}, elliptic estimates give
\[
 \Delta w=O_k(\rho^{-3-2\delta}),\qquad
 w_\xi=O_k(\rho^{-2-2\delta})\quad(\xi=0).
\]
Since \(2\delta>n-2\), the boundary-end argument in
Section~\ref{sec:structure} gives \eqref{boundary-exp},
with \(d=0\) by symmetry.
Reflection gives the lower end, and
\(|\nabla^jh|=O(\rho^{-n-j})\), hence finite total curvature
and extrinsic Euclidean volume growth.

Let \(M_\epsilon\) be the solution just constructed, with contact angle
\(\theta_\epsilon\).
After pullback to \(M_0\), identify the Dirichlet spaces by isometries
tending to the identity. For the compact operators \(K_\epsilon\)
representing \(Q_\epsilon\) as in Section~\ref{sec:classification-three},
the graph estimates give
\[
 \begin{gathered}
 |h_\epsilon|^2\longrightarrow |h_0|^2\quad\text{in }L^{\frac{n}{2}}(M_0),\qquad
 q_\epsilon\longrightarrow q_0\quad\text{in }L^{n-1}(\partial M_0),\\
 \|K_\epsilon-K_0\|\longrightarrow0.
 \end{gathered}
\]
With eigenvalues in decreasing order, the index of \(M_0\) and
\eqref{half-catenoid-kernel} give
\[
 \lambda_1(K_0)>1=\lambda_2(K_0)=\cdots=\lambda_n(K_0)>
 \lambda_{n+1}(K_0).
\]
The min--max principle gives
\(\lambda_1(K_\epsilon)>1>\lambda_{n+1}(K_\epsilon)\) for small \(\epsilon\).
The \(n-1\) translation fields remain independent Robin Jacobi fields of
finite energy, so
\[
 \lambda_2(K_\epsilon)=\cdots=\lambda_n(K_\epsilon)=1,\qquad
 \Ind_{\theta_\epsilon}(M_\epsilon)=1.
\]

Let \(M\) be a normalized index-one hypersurface with
\(\kappa_0=-\cot\theta_0>0\). Set
\[
 \epsilon=\kappa-\kappa_0,\qquad
 T=\nu+\cot\theta_0\,\mu
 \quad\text{on the compact part of }\partial M.
\]
Using Proposition~\ref{kernel}, the preceding argument gives
a smooth family of normalized index-one hypersurfaces \(M_\kappa\)
satisfying
\[
 M_{\kappa_0}=M,\qquad
 |\kappa-\kappa_0|<\epsilon_0,\qquad
 0<\epsilon_0<\kappa_0.
\]
By Proposition~\ref{two-ends}, every nearby normalized index-one
hypersurface has the same \(O(n-1)\) and reflection symmetries.
The uniqueness in the implicit function theorem shows that
\(M_\kappa\) is the only such hypersurface with parameter \(\kappa\).
\end{proof}

\begin{proof}[Proof of Proposition~\ref{hairpin-existence}]
Fix \(n\ge3\). Let \(\mathscr H_n\) be the space of normalized
index-one hypersurfaces with regular ends and \(\kappa>0\),
equipped with the weighted graph topology.
By Propositions~\ref{properness} and~\ref{angle-chart},
\[
 \pi:\mathscr H_n\longrightarrow(0,\infty),
 \qquad \pi(M)=\kappa
\]
is proper and has smooth local inverses.

Let \(\mathscr C_n\) be the connected component containing the
local family constructed near \(C^+\) in
Proposition~\ref{angle-chart}.
This component is open and closed in \(\mathscr H_n\), so
\(\pi|_{\mathscr C_n}\) is proper.
Its image is nonempty, open and closed in \((0,\infty)\); hence
\[
 \pi(\mathscr C_n)=(0,\infty).
\]
A proper local homeomorphism from a connected space onto an
interval is a homeomorphism. Therefore
\[
 \pi^{-1}(\kappa)\cap\mathscr C_n=\{M_\kappa\}
 \qquad(\kappa>0).
\]
The local inverses give smooth dependence on \(\kappa\).
Proposition~\ref{angle-chart} extends this family to \(\kappa=0\),
with \(M_0=C^+\).

For \(\theta\in(0,\pi)\), set \(M_\theta=M_{|\cot\theta|}\)
and choose its unit normal so that
\[
 \nu_{n+1}=\cos\theta\quad\text{on }\partial M_\theta,
 \qquad \Ind_\theta(M_\theta)=1.
\]
For \(\theta\ne\frac{\pi}{2}\), Proposition~\ref{two-ends} gives
\eqref{hairpin-graph}--\eqref{hairpin-graph-equation}.
\end{proof}

\section{Regular ends of capillary surfaces}\label{sec:surface-ends-proof}

In dimension two, we write the Weierstrass representation as
\[
 dX=\operatorname{Re}\Phi,\qquad
 \Phi=\left(\frac{G^{-1}-G}{2},\frac{i(G^{-1}+G)}2,1\right)\omega,
 \qquad \omega=2\partial x_3.
\]
The contact angle gives \(|G|=k\) on \(\partial M\), for a fixed \(k>0\).

\begin{proof}[Proof of Lemma~\ref{surface-ends}]
Finite total curvature and conformal compactification follow from
\cite[Theorem 1.4]{HS23}.
At a boundary end, choose a conformal coordinate
\(z=x+iy\) on \(\{0<|z|<1,\ y\ge0\}\).
Schwarz reflection gives
$
 G(z)=\frac{k^2}{\overline{G(\bar z)}}\quad(y<0).
$
Finite total curvature implies
\[
 \int_{0<|z|<1}
 \frac{4|G'(z)|^2}{(1+|G(z)|^2)^2}\,dx\,dy
 \le C_k\int_E|h|^2\,dA<\infty.
\]
By Picard's theorem and the area formula, \(G\) is meromorphic at \(0\).
Since \(|G(t)|=k\) for real \(t\ne0\),
$
 G(0)\in\mathbb C\setminus\{0\},|G(0)|=k.
$
By \cite[Theorem 6.15]{ABR01},
\[
 \begin{gathered}
 x_3(z)=a\frac{\operatorname{Im}z}{|z|^2}+H_0(z),
 \qquad a\ge0,\\
 \Delta H_0=0,\qquad H_0(\bar z)=-H_0(z).
 \end{gathered}
\]
where \(H_0\) extends as an odd harmonic function across \(\{\operatorname{Im}z=0\}\).
If \(a=0\), then the metric is bounded near \(0\), contradicting
completeness. Thus \(a>0\), and
\[
 X(z)=\operatorname{Re}\left(\frac{A}{z}+B\log z\right)+C+O(|z|),
 \qquad A\cdot A=0,\quad A\ne0.
\]
The \(O(|z|)\) term has bounded first derivatives. Hence, for some \(C_0>1\) and
\(0<|z|\) sufficiently small,
\[
 \begin{aligned}
 C_0^{-1}|z|^{-1}\le |X(z)|\le C_0|z|^{-1},
 C_0^{-1}|z|^{-4}\le \frac{dA}{d\operatorname{Re}z\,d\operatorname{Im}z}
 \le C_0|z|^{-4}.
 \end{aligned}
\]

Let \(P\) be the plane spanned by \(\operatorname{Re}A,\operatorname{Im}A\),
and let \(\nu_\infty\) be its unit normal. Set
\(e=\nu_\infty-\langle\nu_\infty,E_3\rangle E_3\ne0\).
The projection \(\Pi_PY=Y-\frac{(Y\cdot\nu_\infty)e}{|e|^2}\) onto \(P\)
along \(e\) preserves \(x_3\). Identifying \(P\) with \(\mathbb C\), define
\(F(w)=\Pi_PX(\frac{1}{w})\). Then
\[
 F(w)=w+O(\log|w|),\qquad DF=I+O(|w|^{-1}).
\]
Any two points of an exterior half-disk can be joined there by a path of
length at most \(\frac{\pi}{2}\) times their distance. Thus, for sufficiently large
\(R\) and \(v,w\) in this half-disk,
\[
 |(F-I)(v)-(F-I)(w)|
 \le\frac\pi2\sup_{|w|>R}|DF(w)-I|\,|v-w|<|v-w|.
\]
Thus \(F\) is injective. Since \(F\) is proper and has degree one at
infinity, its image contains the complement of a compact subset of
\(P\cap\{x_3\ge0\}\).
The end is a single-valued graph \(X=Y+f(Y)e\) over
\(P\cap\{x_3\ge0\}\) outside a compact set, with
\(f=O(\log r)\), \(|Df|=O(r^{-1})\), and curvature \(O(r^{-2})\).

At an interior end, choose a conformal coordinate \(z\) with \(z=0\)
at the added point. Positivity of \(x_3\) gives
\(x_3=-a\log|z|+H_0\), with \(a\ge0\). The Gauss map is meromorphic and
\(\omega\) has at most a simple pole. If \(\Phi\) had at most a simple pole,
we could write
\[
 \Phi=\left(\frac Bz+O(1)\right)dz,\qquad B\in\mathbb C^3.
\]
Single-valuedness of \(X\) gives, for small \(\epsilon>0\),
\[
 0=\operatorname{Re}\int_{|z|=\epsilon}\Phi
  =\operatorname{Re}(2\pi iB)=-2\pi\operatorname{Im}B,
 \qquad B\in\mathbb R^3.
\]
Conformality gives
\[
 \Phi\cdot\Phi=0\quad\Longrightarrow\quad
 B_1^2+B_2^2+B_3^2=0.
\]
Since \(B\) is real, \(|B|^2=0\), so \(B=0\). The metric would then be
bounded near \(0\), contradicting completeness. Thus \(\Phi\) has a pole
of order \(m+1\ge2\), and integration gives
\[
 X(z)=\operatorname{Re}(Az^{-m})+o(|z|^{-m}),
 \qquad m\ge1,\quad A\cdot A=0,\quad A\ne0.
\]
Embeddedness and \cite{Schoen83} give \(m=1\).
The half-space condition implies
\[
 x_3(z)=\operatorname{Re}\frac{A_3}{z}
       +o(|z|^{-1})\ge0
 \quad\Longrightarrow\quad
 A_3=0,\qquad \nu_\infty=\pm E_3.
\]
The end parametrizations give
\[
 \lim_{R\to\infty}
 \frac{\operatorname{Area}(E\cap B_R)}{\pi R^2}
 =
 \begin{cases}
  1, & E\text{ is an interior end},\\
  \frac{1}{2}, & E\text{ is a boundary end}.
 \end{cases}
\]
Hence
\[
 \lim_{R\to\infty}
 \frac{\operatorname{Area}(M\cap B_R)}{\pi R^2}
 =
 \sum_E\lim_{R\to\infty}
 \frac{\operatorname{Area}(E\cap B_R)}{\pi R^2}
 =e_i+\frac{e_b}{2}.
\]
\end{proof}

\end{document}